\documentclass[12pt]{amsart}

\usepackage[T1]{fontenc}
\usepackage{lmodern}
\usepackage{amsmath,amssymb,amsthm,mathtools,mathrsfs}
\usepackage{booktabs}
\usepackage[shortlabels]{enumitem}
\usepackage[margin=1.08in]{geometry}
\usepackage{microtype}
\usepackage[colorlinks=true,linkcolor=blue,citecolor=blue,urlcolor=blue]{hyperref}
\numberwithin{equation}{section}

\newtheorem{theorem}{Theorem}[section]
\newtheorem{proposition}[theorem]{Proposition}
\newtheorem{lemma}[theorem]{Lemma}
\newtheorem{corollary}[theorem]{Corollary}
\theoremstyle{definition}
\newtheorem{definition}[theorem]{Definition}
\theoremstyle{remark}
\newtheorem{remark}[theorem]{Remark}

\newcommand{\C}{\mathbf C}
\newcommand{\Z}{\mathbf Z}

\newcommand{\R}{\mathbf R}
\newcommand{\PP}{\mathbf P}
\newcommand{\cO}{\mathcal O}

\newcommand{\cL}{\mathcal L}

\newcommand{\NS}{\operatorname{NS}}
\newcommand{\irr}{\operatorname{irr}}

\newcommand{\Pic}{\operatorname{Pic}}
\newcommand{\rk}{\operatorname{rk}}

\newcommand{\Spec}{\operatorname{Spec}}
\newcommand{\divisor}{\operatorname{div}}

\newcommand{\Sym}{\operatorname{Sym}}
\newcommand{\tr}{\operatorname{tr}}
\newcommand{\sat}{\operatorname{sat}}
\newcommand{\indp}{\operatorname{ind}_{+}}

\title[Degree of irrationality of abelian surfaces]
{Degree of Irrationality of Abelian Surfaces}
\author{Songsong Huang, Zhi Jiang, Xin Yang}

\address{\rm Shanghai Center for Mathematical Sciences, Fudan University, Shanghai 200438, China}
\email{sshuang22@m.fudan.edu.cn}

\address{\rm Shanghai Center for Mathematical Sciences, Fudan University, Shanghai 200438, China}
\email{zhijiang@fudan.edu.cn} 
\address{\rm Shanghai Center for Mathematical Sciences, Fudan University, Shanghai 200438, China}
\email{23110840016@m.fudan.edu.cn}
\date{}

\begin{document}

\begin{abstract}
Let $A$ be a complex abelian surface.  We prove that $A$ admits a
dominant rational map of degree three to $\PP^2$ if and only if it
admits a polarization of type $(1,1)$, $(1,2)$, $(1,3)$, or $(1,6)$.
Consequently, the degree of irrationality of $A$ is $3$ precisely in
these four cases and is $4$ otherwise. 
\end{abstract}

\maketitle
\setcounter{tocdepth}{1}
\tableofcontents

\section{Introduction}

For an irreducible complex projective variety $X$ of dimension $n$, its
\emph{degree of irrationality}  
\[
  \irr(X):=\min\left\{\deg(f)\;\middle|\;
  f:X\dashrightarrow\PP^n\text{ is dominant}\right\}.
\]
 was introduced by Heinzer and Moh \cite{HeinzerMoh}.
For an abelian variety $A$,  Alzati and Pirola 
 proved that $\irr(A)>\dim A$ by considering the holomorphic length
\cite{AlzatiPirola}.  For a very general abelian variety of dimension
$g\geq3$, Colombo, Martin, Naranjo, and Pirola  proved the stronger lower
bound $\irr(A)\geq(3g+1)/2$ \cite[Corollary~1.2]{CMNP}.

 Nathan Chen provided the
upper bound $\irr(A)\leq 4$ for a very general polarized abelian surface
\cite[Theorem~1.1]{Chen}, and Chen--Stapleton proved that it specializes
to every complex abelian surface \cite[Corollary~D]{ChenStapleton}. Thus we have
\begin{equation}\label{eq:universal-bounds}
                         3\leq\irr(A)\leq4,
\end{equation}
for any abelian surface.
 
Tokunaga and Yoshihara showed that a smooth genus-three curve on an
abelian surface produces a cubic rational map to $\PP^2$
\cite{TokunagaYoshihara}.  Yoshihara also studied products of elliptic
curves \cite{YoshiharaProduct}.


Martin introduced the decisive cohomological obstruction.  He proved
that if the image of the intersection pairing
$\Sym^2\NS(A)\to\Z$ does not contain $12$, then $\irr(A)=4$
\cite[Theorem~1.1]{Martin}.  In particular, a very general
$(1,d)$-polarized abelian surface has degree of irrationality four when
$d\nmid6$ \cite[Corollary~1.2]{Martin}.  Martin then asked for the
degree of irrationality of very general abelian surfaces of types $(1,1)$,
$(1,3)$, and $(1,6)$ and of a product of two non-isogenous elliptic
curves \cite[Questions~1.4]{Martin}.  Moretti settled the general
$(1,6)$ case using good-pairs correspondence
\cite[Theorem~C(1) and Corollary~3.2]{Moretti}.  While the present paper was being revised, Yongnam Lee announced that
$\irr(E\times F)=3$ for every pair of complex elliptic curves
\cite[Theorem~1.1]{Lee}.

If $L$ is ample on an abelian surface, its polarization type is written
$(d_1,d_2)$ with $d_1\mid d_2$.  Our main theorem removes all
genericity and Picard-number assumptions and sharpens Martin's necessary
obstruction to a complete classification.

\begin{theorem}\label{thm:main}
Let $A$ be a complex abelian surface.  There exists a dominant rational
map
\[
                         \varphi:A\dashrightarrow\PP^2
\]
of degree three if and only if $A$ admits a polarization of type
$(1,1)$, $(1,2)$, $(1,3)$, or $(1,6)$.
\end{theorem}

 We have  the complete answer for the degree of irrationality of  abelian surfaces combining Theorem \ref{thm:main} with \eqref{eq:universal-bounds}.

\begin{corollary}\label{cor:irr-classification}
Let $A$ be a complex abelian surface.  If $A$ admits a polarization of
type $(1,1)$, $(1,2)$, $(1,3)$, or $(1,6)$, then $\irr(A)=3$.
Otherwise $\irr(A)=4$.
\end{corollary}

We sketch the proof of Theorem~\ref{thm:main}.

We first prove the necessary condition by analyzing the residual correspondence associated with a cubic map.  This yields a refinement of Martin's numerical obstruction in the N\'eron--Severi lattice.  We then treat the possible Picard numbers $\rho(A)$ separately.  For $\rho(A)=1$, Martin's theorem gives the result directly \cite[Theorem~1.1]{Martin}.  For $\rho(A)=2$ and $3$, lattice-theoretic arguments produce a primitive polarization of the required types.  Finally, when $\rho(A)=4$, the Shioda--Mitani classification shows that $A$ is a product of isogenous CM elliptic curves and hence admits a principal polarization \cite{ShiodaMitani}; see Section~\ref{sec:necessary}.

For the sufficient direction, products of elliptic curves are treated more directly by
an explicit order-nine tangential-triangle construction (see Section~\ref{sec:common} and also \cite{Lee}).  For an indecomposable principal
polarization, we choose a \emph{good} $3$-torsion point (Definition~\ref{def:good-torsion}) and
use six translates of the smooth genus-two theta curve to construct an
equivariant projective two-dimensional linear system of degree three (see Section~\ref{sec:11}).

For a polarized abelian surface $(A, L)$ of type $(1,2)$,
if $|L|$ has no fixed component, it contains a smooth genus-three curve and thus we can apply Tokunaga--Yoshihara's theorem \cite{TokunagaYoshihara}. While $|L|$ has a
fixed component, $A$ is isomorphic to a product of two elliptic curves  (see Section~\ref{sec:12}).  

For a polarized abelian surface $(A, L)$ of type $(1,3)$, we pass to an \'etale $\mathbf Z_3$-cover $B$ with a principal polarization equipped with a marked $3$-torsion point. If the marking is good, the preceding construction yields an equivariant cubic rational map
$B\dashrightarrow \PP^2$,
which descends to a cubic rational map
$
A\dashrightarrow \PP^2.
$
If the marking is not good, we deform the marked principally polarized surface, together with its full level-$3$ structure, to the good locus, extend the corresponding equivariant linear system over a discrete valuation ring, and then specialize it back to the original surface; see Section~\ref{sec:13}.


For type $(1,6)$, following
Moretti \cite{Moretti}, we construct a \emph{good pair} and consider its corresponding \emph{determinant linear system} on $A$; see Subsection~\ref{subsec:moretti-good-pairs} for these concepts.
We treat both a
possible fixed divisor of the determinant system and the possibility
that its image is a curve; the latter cases reduce to a product or to
type $(1,2)$; see Section~\ref{sec:16}.\\

We also record that the Alzati--Pirola bound is sharp in dimension three.

\begin{remark}\label{rem:j1728-sharp}
Let $E$ be the elliptic curve $y^2=x(x^2-1)$, which has $j(E)=1728$, and let
$\iota(x,y)=(-x,iy)$ be its order-four automorphism.  The diagonal action
of $G=\langle\iota\rangle$ on $E^3$ is generically free, since every
nontrivial power of $\iota$ has only finitely many fixed points.  Hence
$E^3\to E^3/G$ has degree four.  Catanese--Oguiso--Truong identify the
quotient birationally with an explicit conic bundle
\cite[Proposition~2.5 and Corollary~2.6]{COT14}, and Colliot-Th\'el\`ene proves that this
conic bundle is rational \cite[\S3]{CT15}.  Composing with a birational
map $E^3/G\dashrightarrow\PP^3$ gives $\irr(E^3)\leq4$.  Conversely,
$E^3$ has holomorphic length three, so the Alzati--Pirola inequality
\cite[Theorem~(3,4)]{AlzatiPirola} gives degree at least four for every
dominant rational map $E^3\dashrightarrow\PP^3$.  Thus
$\irr(E^3)=4$.
\end{remark}

It would be very interesting to establish a version of Theorem \ref{thm:main} in dimension $3$.

\subsection*{Acknowledgements}
AI tools were used to assist with exploratory work, including suggesting examples, carrying out auxiliary calculations, and reviewing drafts. The overall proof strategies and mathematical arguments were developed under the authors' guidance. The authors carefully checked and rederived the entire manuscript, including all AI-assisted arguments and references, and take full responsibility for the final text.

 The authors are supported by the Foundation for Innovative Research Groups of the Natural Science Foundation of China (No. 12121001). The second author is a member of the Key Laboratory of Mathematics for
Nonlinear Science, Fudan University and is also supported by the natural Science Foundation of China (No. 12371044), by the Shanghai Leading Talent Program of Eastern Talent Plan, and by the Natural Science Foundation of Shanghai (No. 23ZR1404000).
\section{Preliminaries}\label{sec:prelim}

\subsection{Conventions}
\label{subsec:conventions}

All abelian varieties are defined over $\C$, and their group laws are
written additively.  The origin of an abelian variety $A$ is denoted
$0_A$, multiplication by an integer $n$ is denoted
$[n]_A:A\to A$, and $A[n]:=\ker([n]_A)$.  If $n\neq0$, then $[n]_A$ is
an  isogeny of degree $n^{2\dim A}$ and
$A[n]\simeq(\Z/n\Z)^{2\dim A}$.  By an \emph{isogeny} we mean a
surjective homomorphism between abelian varieties with a finite kernel.  Thus,
for every finite subgroup $G\subset A$, the quotient homomorphism
$A\to A/G$ is an  isogeny of degree $|G|$.
If $q:B\to A$ is such an isogeny and $N$ is a line bundle on $B$, finite
flat base change gives
\[
  q^*q_*N\simeq\bigoplus_{a\in\ker(q)}T_a^*N,
\]
while $q_*q^*=(\deg q)\operatorname{id}$ on N\'eron--Severi groups.  In
particular, on an abelian surface both $[n]_A^*$ and $[n]_{A*}$ act as
multiplication by $n^2$ on $\NS(A)$.

For a line bundle $L$ on $A$, the theorem of the square shows that the
map
\[
  \phi_L:A\longrightarrow\widehat A,
  \qquad x\longmapsto T_x^*L\otimes L^{-1},
\]
is a homomorphism.  The line bundle $L$ is nondegenerate precisely when
$\phi_L$ is an isogeny, and we write $K(L):=\ker(\phi_L)$.  We use
$T_x:A\to A$, $z\mapsto z+x$, for translation.  If $D\subset A$ is a
divisor, our convention is $D+x:=T_x(D)$; with this convention, we have
\begin{equation}
  \cO_A(D+x)\simeq T_{-x}^*\cO_A(D).
\label{eq:translate-convention}
\end{equation}
\nopagebreak[4]

Weil's extension theorem
says that every rational map from a smooth variety to an abelian variety
extends uniquely to a morphism; in particular, every birational self-map
of an abelian variety is a regular automorphism.  A rational map from a
smooth projective rational variety to an abelian variety is constant,
because the source has trivial Albanese variety.  We use these facts in
the forms recorded in \cite[Theorem~4.9.4 and Proposition~4.9.5]{BL}.

We write $D\equiv D'$ for numerical equivalence, and we set
$\operatorname{Stab}(D):=\{x\in A\mid D+x=D\}$.  The symbol $\boxtimes$
denotes the exterior tensor product.

In our paper, $\PP(W)$
parametrizes one-dimensional subspaces of a vector space $W$. For a real symmetric form $R$, the notation $R\preceq0$ means that $R$
is negative semidefinite, and $\indp(R)$ denotes its positive index.

\subsection{Positivity on abelian surfaces}
Let $A$ be an abelian surface and put $\Lambda=\NS(A)$.  The
intersection form on $\Lambda$ is even, since Riemann--Roch gives
\begin{equation}\label{eq:even-intersection}
                         D^2=2\chi(A,D)
                         \qquad(D\in\NS(A)),
\end{equation}
see \cite[Theorem~3.6.3]{BL}.  The Hodge index theorem gives signature
$(1,\rho(A)-1)$.  The group $\NS(A)$ is free
\cite[Corollary~5.3.9]{BL}.
Choose an integral basis $e_1,\dots,e_r$ of $\Lambda$, where
$r=\rho(A)$, and write
\[
                         Q=(e_i\cdot e_j)_{i,j=1}^r.
\]
We use column coordinates.  Thus $Q$ represents a linear map
$\Lambda\to\Lambda^*$, while a symmetric tensor
$B\in \Sym^2\Lambda$ is represented by a
symmetric integral matrix and defines a map
$B:\Lambda^*\to\Lambda$.  Consequently $BQ$ is an
endomorphism of $\Lambda$ and $QBQ$ is a symmetric form on $\Lambda$.

If an ample line bundle $L$ is of type $(d_1,d_2)$ on an abelian surface, then
$L^2=2d_1d_2$, $h^0(A,L)=d_1d_2$, and
$|K(L)|=(d_1d_2)^2$.
We use the following facts repeatedly; standard references are
\cite{MumfordAV,BL}.

\begin{lemma}\label{lem:standard}
Let $S$ be a complex abelian surface.
\begin{enumerate}[(i)]
\item Every effective divisor on $S$ is nef.
\item If $N$ is nef and $N^2>0$, then $N$ is ample,
      $H^i(S,N)=0$ for $i>0$, and
      $h^0(S,N)=\chi(N)=N^2/2$.
\item If $N$ is nonzero and effective with $N^2=0$, there are an elliptic
      quotient $q:S\to C$,
      a line bundle $N_0$ of positive degree $r$ on the elliptic curve
      $C$, and an elliptic fiber class $F$ such that
      $N\simeq q^*N_0$, $N\equiv rF$, and $h^0(S,N)=r$.
\item If $H$ is a polarization and $F\subset S$ is an elliptic curve with
      $H\cdot F=1$, then $S$ is isomorphic to a product of two elliptic
      curves.
\end{enumerate}
\end{lemma}

\begin{proof}
For (i), if $D$ is effective and $C\subset S$ is an irreducible curve,
translate $D$ so that it does not contain $C$; translation preserves its
numerical class, and the resulting intersection with $C$ is nonnegative.
For (ii), a nef line bundle of positive top self-intersection on an abelian
variety is ample.  Kodaira vanishing and Riemann--Roch on $S$ then give the
displayed formula.

For (iii), write an effective divisor representing $N$ as
$\sum n_iC_i$.  Since every term $C_i^2$ and $C_i\cdot C_j$ is
nonnegative and the total square is zero, all these intersections vanish.
Adjunction and the absence of rational curves on an abelian variety imply
that every $C_i$ is a smooth elliptic curve.  Curves among the $C_i$ are
translates of a single elliptic subgroup; otherwise two distinct elliptic
directions would have positive intersection.  Hence the divisor is a sum
of fibers of the connected-fiber quotient $q:S\to C$, and it is
$q^*N_0$ for a degree-$r$ line bundle $N_0$ on $C$.  Since
$q_*\cO_S=\cO_C$, the projection formula gives
$h^0(S,N)=h^0(C,N_0)=r$.

For (iv), translate $F$ inside $S$ so that $0_S\in F$; this does not
change $H\cdot F$.  By the rigidity lemma, the inclusion
$i:F\hookrightarrow S$ is then a homomorphism.  Let
$\widehat i:\widehat S\to\widehat F$ be its dual.  Since $H|_F$ has
degree one, $\phi_{H|_F}:F\to\widehat F$ is an isomorphism.  The
homomorphism
$r=(\phi_{H|_F})^{-1}\circ\widehat i\circ\phi_H:S\to F$ satisfies
$r\circ i=\mathrm{id}_F$.  The maps
$x\mapsto(r(x),x-i(r(x)))$ from $S$ to $F\times\ker(r)$ and
$(y,z)\mapsto i(y)+z$ in the other direction are inverse homomorphisms.
Since $S$ is connected, $\ker(r)$ is connected and is an elliptic curve.
\end{proof}

We shall also use repeatedly the following sharper description of
polarizations of type $(1,d)$ on an abelian surface.

\begin{lemma}\label{lem:BL-1d}
Let $L$ be an ample line bundle of type $(1,d)$, where $d\geq2$, on an
abelian surface $S$.
\begin{enumerate}[(i)]
\item The linear system $|L|$ has a fixed component if and only if there
are elliptic curves $E_1,E_2$, line bundles $L_i$ with
$\deg L_1=1$ and $\deg L_2=d$, and an isomorphism of polarized abelian
varieties
$(S,L)\simeq(E_1\times E_2,p_1^*L_1\otimes p_2^*L_2)$.
\item If $|L|$ has no fixed component, then it is base-point-free for
$d\geq3$, while for $d=2$ its base scheme consists of four points.  A
general member is smooth when $d\geq2$.  When $d=2$ and $L$ is symmetric, the four base
points are four-division points.
\end{enumerate}
\end{lemma}

\begin{proof}
Part (i) is \cite[Lemma~10.1.1]{BL}.  The base-locus statements are
\cite[Lemma~10.1.2]{BL}, smoothness is
\cite[Proposition~10.1.3]{BL}, and the last assertion is
\cite[Example~10.1.4]{BL}.
\end{proof}



We shall use the following positivity consequences.
\begin{lemma}\label{lem:positive-is-ample}
If
$0\neq D\in\NS(A)$ satisfy $D^2>0$.  After replacing $D$ by $-D$, the
class $D$ is ample.
\end{lemma}

\begin{proof}
Refer to \cite[Corollary~4.3.3]{BL}.
\end{proof}

\begin{lemma}\label{lem:primitive-type}
If $D\in\NS(A)$ is primitive and $D^2=2d>0$, then one of $D,-D$ is the
class of a polarization of type $(1,d)$.
\end{lemma}

\begin{proof}
By saturation, $D$ is primitive in integral cohomology.  If an ample
line bundle has type $(d_1,d_2)$, then its square is $2d_1d_2$, and the
divisibility of its alternating Riemann form is $d_1$.  Its first Chern
class is therefore primitive exactly when $d_1=1$; see
\cite[Section~3.1 and Lemma~3.6.4]{BL}.  Apply
Lemma~\ref{lem:positive-is-ample}.
\end{proof}


\subsection{Martin's cohomological obstruction}
\label{subsec:martin-obstruction}

Let $I(A)\subset\Z$ denote the image of the intersection pairing
\[
                 \Sym^2\NS(A)\longrightarrow\Z,
                 \qquad D\odot E\longmapsto D\cdot E.
\]
The following theorem is the starting point for the necessary direction.

\begin{theorem}[Martin]\label{thm:martin}
If a complex abelian surface admits a dominant rational map of degree
three to $\PP^2$, then $12\in I(A)$.
\end{theorem}

\begin{proof}
It directly follows from \cite[Theorem~1.1]{Martin}; see also Proposition~\ref{prop:residual-class}.
\end{proof}

\begin{corollary}\label{prop:rank-one-necessity}
Let \(A\) be a complex abelian surface with \(\rho(A)=1\).  If \(A\)
admits a dominant rational map of degree three to \(\PP^2\), then it
admits a polarization of type \((1,d)\) with
\(d\in\{1,2,3,6\}\).
\end{corollary}

\begin{proof}
Let \(h\) be a primitive generator of \(\NS(A)\), with its sign chosen
so that it is ample.  Write \(h^2=2d\).  The image of the intersection
pairing is \(I(A)=2d\Z\).  Theorem~\ref{thm:martin} gives
\(12\in2d\Z\), so \(d\mid6\).  Lemma~\ref{lem:primitive-type}
finishes the proof.
\end{proof}

We recall the constant-cycle construction in Martin's
proof of Theorem~\ref{thm:martin}
\cite[proof of Theorem~1.1]{Martin}: instead of retaining only one
intersection number, we keep the full integral tensor on
\(\NS(A)\).

Let $\varphi:A\dashrightarrow\PP^2$
be dominant of degree three.  
We may assume that every general
fiber \(\{x,y,z\}\) satisfies
\begin{equation}\label{eq:fiber-sum}
                              x+y+z=0.
\end{equation}
Define the \emph{reduced ordered residual correspondence} by
\begin{equation*}\label{eq:def-W}
 W=\overline{\{(x,y)\in A\times A:x\neq y,
                    \ \varphi(x)=\varphi(y)\}}\subset A\times A,
\end{equation*}
where the closure is taken from the finite \'etale locus of $\varphi$.  Both
projections $W\to A$ have total generic degree two.

\begin{proposition}\label{prop:residual-class}
Put \(\Lambda=\NS(A)\), let \(e_1,\dots,e_r\) be an integral basis,
where \(r=\rho(A)\), and let \(Q=(e_i\cdot e_j)_{i,j = 1}^r\).  There is an
integral symmetric matrix \(B=(b_{ij})\in\Sym_r(\Z)\) such that, in
\(H^4(A\times A,\Z)\),
\begin{equation}\label{eq:residual-class}
 [W]=3[A\times\{0\}]+3[\{0\}\times A]-[\Delta_A]+\Gamma_B,
 \qquad
 \Gamma_B=\sum_{i,j=1}^r b_{ij}e_i\boxtimes e_j.
\end{equation}
Moreover,
\begin{equation}\label{eq:trace-twelve}
                              \tr(BQ)=12.
\end{equation}
We shall call \(B\) the \emph{residue matrix}.
\end{proposition}

\begin{proof}
Refer to \cite[Proof of Theorem~1.1 and Proposition 2.1]{Martin}.
\end{proof}

\subsection{Moretti's good-pairs correspondence}
\label{subsec:moretti-good-pairs}

We recall the part of Moretti's construction used in Section~\ref{sec:16}.
Let $X$ be a smooth projective variety of dimension $n$, let $\mathcal E$
be a reflexive sheaf of rank $n$, and let
$V^\vee\subset H^0(X,\mathcal E)$ be a vector space of dimension $n+1$.
Following \cite[Definition~1.3]{Moretti}, the pair
$(\mathcal E,V^\vee)$ is called a \emph{good pair} if $V^\vee$ generates
$\mathcal E$ in codimension one and
$H^0(X,\mathcal E^\vee)=0$.

Given a good pair $(\mathcal E,V^\vee)$, the kernel of the evaluation
map $\operatorname{ev}:V^\vee\otimes\mathcal O_X\to\mathcal E$ is
$\mathcal L^{-1}$, where $\mathcal L:=\det\mathcal E$.  Dualizing
produces an exact sequence
\begin{equation*}
  0\longrightarrow\mathcal E^\vee\longrightarrow
  V\otimes\mathcal O_X\longrightarrow\mathcal L.
\end{equation*}
Thus $V$ is a subspace of $H^0(X,\mathcal L)$, and
\[
 V=\operatorname{im}\!\left(
       \bigwedge\nolimits^n V^\vee
       \longrightarrow H^0(X,\mathcal L)\right).
\]
The linear subsystem $|V|\subset|\mathcal L|$, called the \emph{determinant linear system}, induces a nondegenerate
rational map
\[
  \varphi_V:X\dashrightarrow\PP(V^\vee).
\]
At a point where evaluation is surjective, the map sends $x$ to the
one-dimensional kernel of $V^\vee\to\mathcal E_x$; we therefore call
$\varphi_V$ the \emph{kernel map} of $(\mathcal E,V^\vee)$.  If
$[s]\in\PP(V^\vee)$, then
\begin{equation}\label{fiber_formula}
  \varphi_V^{-1}([s])\cap\bigl(X\setminus\operatorname{Bl}(V^\vee)\bigr)
  =Z_{\mathcal E}(s)\cap
   \bigl(X\setminus\operatorname{Bl}(V^\vee)\bigr),
\end{equation}
where $\operatorname{Bl}(V^\vee)$ is the locus where $\mathcal E$ is not
locally free or is not generated by $V^\vee$, and $Z_{\mathcal E}(s)$
is the zero locus of $s$ as a section of $\mathcal E$.

Moretti's correspondence asserts that every good pair determines this
nondegenerate rational map and that, when the map is generically finite,
its degree is the number of zeros away from
$\operatorname{Bl}(V^\vee)$ of a general section of $V^\vee$, counted
with multiplicity \cite[Proposition~1.5]{Moretti}.  Consequently, if
$\varphi_V$ is generically finite, $\mathcal E$ is a vector bundle, and
a general section has a
zero-dimensional zero scheme, then one obtains the useful estimate
\[
  \deg(\varphi_V)\leq c_n(\mathcal E).
\]
The correspondence does not by itself guarantee that $\varphi_V$ is
generically finite; this is the point that must be checked separately in
the type-$(1,6)$ argument.

\subsection{Full level-three structures and relative theta divisors}
\label{subsec:level-structures}

Fix a primitive third root of unity and the standard symplectic form on
$(\Z/3\Z)^4$.  A full symplectic level-$3$ structure on a principally
polarized abelian surface $(B,\Theta)$ is an isomorphism
$\alpha:(\Z/3\Z)^4\xrightarrow{\sim}B[3]$ carrying the standard form to
the Weil pairing of $\Theta$.  For level at least three, Mumford's
construction represents the moduli functor by a quasi-projective fine
moduli scheme, which therefore carries a universal principally polarized
abelian scheme $\pi:\mathscr B\to\mathcal A_2(3)$
\cite[Chapter~7, Theorems~7.9--7.10]{MFK}.  Over $\C$, the
analytification of the chosen symplectic component is
$\Gamma(3)\backslash\mathfrak H_2$ \cite[Theorem~8.3.2]{BL}.  The
action is free \cite[Corollary~5.1.10]{BL}, so this component is smooth;
it is connected, and hence irreducible.
The group
$\operatorname{Sp}_4(\mathbf F_3)$ acts by changing the marking.  This
action changes only the labels of the three-torsion points: it does not
alter the underlying abelian surface, its polarization, or a chosen theta
divisor.  Since every nonzero vector in a symplectic vector space is
isotropic, elementary symplectic linear algebra shows that
$\operatorname{Sp}_4(\mathbf F_3)$ acts transitively on the nonzero
vectors of $(\mathbf F_3)^4$.

Set $S=\mathcal A_2(3)$.  A polarization on the universal abelian scheme
$\pi:\mathscr B\to S$ is initially a
homomorphism to the dual scheme rather than a specified line bundle.
The symmetric line bundles inducing it form a torsor under the finite
\'etale group scheme $\widehat{\mathscr B}[2]$ (identified with
$\mathscr B[2]$ by the principal polarization)
\cite[Lemma~4.6.2]{BL}.  Consequently, after a finite \'etale base
change, we replace $S$ and $\mathscr B$ by their pullbacks and retain the
same notation.  On this cover one may choose a symmetric relatively
ample line bundle $\mathscr L$ and rigidify it along the zero section.
Since $\pi_*\mathscr L$ is invertible and formation of global sections
commutes with base change, the evaluation morphism
$\pi^*\pi_*\mathscr L\to\mathscr L$ cuts out a relative effective Cartier
divisor $\mathscr C\subset\mathscr B$; explicitly,
$\cO_{\mathscr B}(\mathscr C)=
\mathscr L\otimes\pi^*(\pi_*\mathscr L)^{-1}$.  Thus rigidification is
additional structure on the theta line bundle, while $\mathscr C$ is its evaluation
divisor inside the total abelian scheme.  The finite \'etale cover records
only the choice of a symmetric representative and does not modify any
geometric fiber.  Two representatives differ by a relative two-torsion
translation, so they have the same image under multiplication by two.

\section{The necessary condition}\label{sec:necessary}
\subsection{The residual matrix of a cubic map}
\label{sec:residual}

We now prove the necessary direction of Theorem~\ref{thm:main}.  We keep the notation and assumption in Subsection~\ref{subsec:martin-obstruction}.
Put \(K=\C(\PP^2)\) and \(L=\C(A)\).  The extension \(L/K\) is
separable of degree three.

\begin{proposition}\label{prop:cyclic-low-rank}
If \(\rho(A)\le3\), then \(L/K\) is not cyclic.
\end{proposition}

\begin{proof}
Suppose that \(L/K\) is cyclic, with deck transformation \(\sigma\).
The birational map \(\sigma:A\dashrightarrow A\) extends to a regular
automorphism.  Over the finite \'etale locus and hence after taking
closures, as cycles, we have
\[
                         [W]=[\Gamma_\sigma]+[\Gamma_{\sigma^2}],
\]
where $\Gamma_\sigma$ and $\Gamma_{\sigma^2}$ are the graphs of $\sigma$ and $\sigma^2$, respectively.
Let \(R\) be the matrix of the linear transformation $\sigma^*|_{\Lambda_{\R}}$ with respect to the basis chosen in Proposition~\ref{prop:residual-class}.  Comparing intersections of
the two sides with \eqref{eq:residual-class} gives
\[
                         BQ=I+R+R^2.
\]
Indeed, the matrix of the two graphs is \(Q(R+R^2)\), whereas
the matrix in \eqref{eq:residual-class} is \(QBQ-Q\).
Since \(R^3=I\), the endomorphism
\((I+R+R^2)/3\) is the projector onto \(\Lambda_{\R}^R\).  Hence
\[
 12=\tr(BQ)=3\dim_{\R}\Lambda_{\R}^R\le3\rho(A)<12,
\]
a contradiction.
\end{proof}

\begin{remark}\label{rem:cyclic-rank-four}
The restriction on the Picard number in
Proposition~\ref{prop:cyclic-low-rank} is essential.  Cyclic cubic maps
can occur when \(\rho(A)=4\).  For example, let
$E=\C/(\Z+\Z\zeta_3)$ and let
$\sigma=[\zeta_3,\zeta_3]\in\operatorname{Aut}(E\times E)$.
Yoshihara's classification shows that
$(E\times E)/\langle\sigma\rangle$ is rational
\cite[Theorem~2.1, pp.~136--137, and Note~2.5, p.~138]
{YoshiharaQuotients}.  Thus the quotient, followed by a birational map
to $\PP^2$, is a cyclic cubic map.  This causes no exception to the main
theorem, because every rank-four complex abelian surface is principally
polarizable by Lemma~\ref{lem:rho-four-principal} below.
\end{remark}

Assume now that \(L/K\) is non-Galois, and let \(N/K\) be its normal
closure.  Then \(\operatorname{Gal}(N/K)\simeq S_3\), and, as
\(L\)-algebras,
\[
                         L\otimes_KL\simeq L\times N.
\]
Indeed, after choosing a primitive element, its cubic minimal polynomial
over \(K\) factors over \(L\) as its linear factor times an irreducible
quadratic factor; reducibility of the latter would make \(L/K\) normal.
The generic residual algebra is therefore the field \(N\), so \(W\) is
integral and both projections \(W\to A\) have degree two.

Let \(Y\) be the normalization of \(\PP^2\) in \(N\), and choose a
projective \(S_3\)-equivariant resolution \(S\to Y\).  The three
embeddings of \(L\) into \(N\) give rational maps \(S\dashrightarrow A\),
which extend by Weil's theorem to morphisms
\[
                         q_i:S\longrightarrow A,\qquad i=1,2,3.
\]
Each has degree two.  Two distinct cubic subfields generate \(N\);
equivalently, their distinct order-two stabilizers in \(S_3\) have
trivial intersection.  Thus \((q_i,q_j): S\to A\times A\) is birational onto \(W\)
for \(i\ne j\), and
\begin{equation}\label{eq:pair-push}
                         (q_i,q_j)_*[S]=[W].
\end{equation}

\begin{proposition}\label{prop:hodge-tensors}
In the non-Galois case, the residual matrix $B$ satisfies
\begin{equation}\label{eq:hodge-tensor-constraints}
 \indp(B)=1,\qquad
 3Q-QBQ\preceq0.
\end{equation}
Moreover, \(Q+QBQ\) has positive index one.  In particular, if
\(\rho(A)=2\), then
\begin{equation}\label{eq:rank-two-plus-determinant}
                         \det(Q+QBQ)\le0.
\end{equation}
\end{proposition}

\begin{proof}
For \(D,E\in\Lambda_{\R}\), the projection formula and
\eqref{eq:pair-push} give
\begin{align}
 (q_i^*D\cdot q_i^*E)&=2D^tQE,\label{eq:self-pairing}\\
 (q_i^*D\cdot q_j^*E)&=D^t(QBQ-Q)E
                         \qquad(i\ne j).\label{eq:cross-pairing}
\end{align}
Set \(T(D)=q_1^*D+q_2^*D+q_3^*D\).  By \eqref{eq:self-pairing} and \eqref{eq:cross-pairing}, we have
\begin{equation}\label{eq:sum-gram}
                         (T(D)\cdot T(E))=6D^tQBQE.
\end{equation}
If \(H\) is ample, then \(T(H)\) is nef and big.  The Hodge index
theorem on \(S\), restricted to the image of \(T\), shows that \(QBQ\)
has at most one positive direction; \eqref{eq:sum-gram} at \(H\)
shows that it has one.  Since \(QBQ\) is congruent to \(B\),
\(\indp(B)=1\).

For the remaining constraints, put
\[
 U_+(D)=q_1^*D+q_2^*D,\qquad
 U_-(D)=q_1^*D-q_2^*D.
\]
Equations \eqref{eq:self-pairing}--\eqref{eq:cross-pairing} give
\begin{align}
 (U_+(D)\cdot U_+(E))&=2D^t(Q+QBQ)E,\label{eq:plus-gram}\\
 (U_-(D)\cdot U_-(E))&=2D^t(3Q-QBQ)E,\label{eq:minus-gram}\\
 (U_+(D)\cdot U_-(E))&=0.\label{eq:plus-minus-orthogonal}
\end{align}
The class \(U_+(H)\) is nef and big. Its orthogonal complement is
negative definite, so \eqref{eq:minus-gram} gives
\(3Q-QBQ\preceq0\).  The same Hodge-index argument applied to the
image of \(U_+\), together with \(U_+(H)^2>0\), shows that
\(Q+QBQ\) has positive index one.  In rank two this implies
\eqref{eq:rank-two-plus-determinant}.
\end{proof}

\subsection{Completion of the necessary condition}

\begin{lemma}\label{lem:rho-four-principal}
Let \(A\) be a complex abelian surface with \(\rho(A)=4\).  Then \(A\)
admits a principal polarization.
\end{lemma}

\begin{proof}
An abelian surface of Picard number four is singular in the terminology
of Shioda--Mitani.  Their classification gives elliptic curves
\(E_1,E_2\) and an isomorphism of complex abelian varieties
\[
                         f:A\xrightarrow{\sim}E_1\times E_2;
\]
see \cite[p.~259 and Theorem~3.1]{ShiodaMitani}.  More precisely,
the two factors may be chosen mutually isogenous and with complex
multiplication.  If \(p_i\) are the projections, then
\[
 \cL_{\mathrm{prod}}
 =p_1^*\cO_{E_1}(0_{E_1})\otimes
  p_2^*\cO_{E_2}(0_{E_2})
\]
is the product principal polarization.  Thus
\(f^*\cL_{\mathrm{prod}}\) is a polarization of type \((1,1)\) on
\(A\).
\end{proof}

\begin{theorem}\label{thm:necessary}
Let \(A\) be a complex abelian surface.  If there is a dominant
rational map \(A\dashrightarrow\PP^2\) of degree three, then \(A\)
admits a polarization of type \((1,1)\), \((1,2)\), \((1,3)\), or
\((1,6)\).
\end{theorem}

\begin{proof}
If \(\rho(A)=1\), apply
Corollary~\ref{prop:rank-one-necessity}.  If \(\rho(A)=4\), apply
Lemma~\ref{lem:rho-four-principal}, independently of the map.

Suppose \(\rho(A)=2\) or \(3\).  Proposition~\ref{prop:cyclic-low-rank}
excludes the cyclic case.  In the non-Galois case,
Propositions~\ref{prop:residual-class} and~\ref{prop:hodge-tensors}
provide an integral symmetric matrix \(B\) satisfying the hypotheses of
Lemma~\ref{lem:binary-tensor} in rank two and of
Proposition~\ref{prop:rank-three-arithmetic} in rank three.  In either
case, the corresponding lattice result supplies a primitive class
\(D\in\NS(A)\) with
\[
                         D^2\in\{2,4,6,12\}.
\]
Lemmas~\ref{lem:positive-is-ample} and~\ref{lem:primitive-type}
turn one of \(D,-D\) into one of the four asserted polarizations.
\end{proof}

\section{Two common cubic constructions}\label{sec:common}

\subsection{The cubic construction on a product}

Lee's preprint proves the following proposition by putting the two elliptic
curves into oppositely oriented cyclic plane-cubic normal forms, varying
these forms over the $j$-line, and using their coordinatewise product
sections \cite[Theorem~1.1]{Lee}.  The proof below is
slightly different: an order-$9$ point on each curve produces the required
tangential triangle intrinsically.

\begin{proposition}\label{prop:product}
For arbitrary complex elliptic curves $E$ and $E'$, there exists a
dominant rational map $E\times E'\dashrightarrow\PP^2$ of degree three.
\end{proposition}

\begin{proof}
Embed $E$ in $\PP^2$ by the complete linear system $|3[0_E]|$.
Choose $p\in E$ of exact order nine and put
$P_1=p$, $P_2=-2p$, and $P_3=4p$.
The tangent line at $P_i$ meets the cubic again at $P_{i+1}$, with indices
modulo three, because the third point on the tangent at $x$ is $-2x$ and
$-8p=p$.  The points $P_1,P_2,P_3$ are distinct and noncollinear: their
sum is $3p\neq0$.  After a projective change of coordinates, take them to
the coordinate vertices.  If $x_1,x_2,x_3$ are the corresponding
coordinate sections of $\cO_E(3[0_E])$, then
\begin{equation}
\begin{aligned}
 \divisor(x_1)&=2P_2+P_3,\\
 \divisor(x_2)&=2P_3+P_1,\\
 \divisor(x_3)&=2P_1+P_2.
\end{aligned}
\label{eq:E-orders}
\end{equation}

Perform the same construction on $E'$, with the cyclic tangential triangle
labelled in the opposite orientation.  For vertices $Q_1,Q_2,Q_3$ and
coordinate sections $y_1,y_2,y_3$ of $\cO_{E'}(3[0_{E'}])$, arrange
\begin{equation}
\begin{aligned}
 \divisor(y_1)&=Q_2+2Q_3,\\
 \divisor(y_2)&=2Q_1+Q_3,\\
 \divisor(y_3)&=Q_1+2Q_2.
\end{aligned}
\label{eq:Eprime-orders}
\end{equation}
Consider the linear system
\[
  [x_1\boxtimes y_1:x_2\boxtimes y_2:x_3\boxtimes y_3]
  :E\times E'\dashrightarrow\PP^2.
\]
The three sections lie in
$N:=\cO_E(3[0_E])\boxtimes\cO_{E'}(3[0_{E'}])$, whose square is
$N^2=2\cdot3\cdot3=18$.

The base support is precisely $\{(P_i,Q_j):i\neq j\}$.
Indeed, suppose that $(z,z')\in E\times E'$ is a base point.  If
$z\notin\{P_1,P_2,P_3\}$, then all three $x_i(z)$ are nonzero, because
the divisors in \eqref{eq:E-orders} are supported on the $P_i$.
Consequently, we have $y_1(z')=y_2(z')=y_3(z')=0$, contradicting the
base-point-freeness of $|3[0_{E'}]|$.  It follows that $z=P_i$ for some
$i$.
Since $x_i(P_i)\neq0$, one must have $y_i(z')=0$.  By
\eqref{eq:Eprime-orders}, this means that $z'=Q_j$ for some $j\neq i$.
Conversely, the order tables show that every such off-diagonal pair is a
base point, while at $(P_i,Q_i)$ the $i$-th coordinate is nonzero.
In local parameters $u,v$
on $E,E'$, respectively, the exact base ideals are
\begin{equation}
\begin{array}{c|c}
\text{base points}&\text{base ideal}\\ \hline
(P_1,Q_3),(P_2,Q_1),(P_3,Q_2)&(u^2,uv,v^2)=(u,v)^2,\\
(P_1,Q_2),(P_2,Q_3),(P_3,Q_1)&(u,v).
\end{array}
\label{eq:product-ideals}
\end{equation}

Let $\mu:\widetilde{E\times E'}\to E\times E'$ be the blow-up of these
six points, and denote by $E_{ij}$ the exceptional
curve over $(P_i,Q_j)$.  Because the local ideals in
\eqref{eq:product-ideals} are the full powers $(u,v)^2$ and $(u,v)$, this
single blow-up resolves the linear system with no residual infinitely-near base
point.  Its moving divisor is
$H=\mu^*N-2(E_{13}+E_{21}+E_{32})
-(E_{12}+E_{23}+E_{31})$, and
\[
  H^2=18-3\cdot2^2-3\cdot1^2=3.
\]
The base-point-free divisor $H$ defines a morphism
$\widetilde f:\widetilde{E\times E'}\to\PP^2$.
If its image were a curve, then $H$ would be pulled back from that curve
and $H^2$ would be zero.  Hence the morphism $\widetilde f$ is dominant.
Since
$H=\widetilde f^*\cO_{\PP^2}(1)$, one has
$\deg(\widetilde f)=H^2=3$.
The corresponding rational map from $E\times E'$ has degree three.
\end{proof}

\subsection{The good-torsion theta-product linear system}
\label{subsec:good-torsion}

\begin{definition}\label{def:good-torsion}
Let $(B,\Theta)$ be an indecomposable principally polarized abelian
surface and let $C\subset B$ be a smooth symmetric theta divisor.  A
point $0\ne t\in B[3]$ is \emph{good} (with respect to $C$) if
$t\notin[2]_B(C)$.  Equivalently, there is no $x\in C$ with $2x=t$.
Because two symmetric theta divisors differ by a two-torsion translation,
the condition is independent of the chosen symmetric representative.
\end{definition}

The significance of goodness is that it keeps the six base points in the
construction below distinct and forces exactly three double and three
simple base points.  The proposition therefore supplies not merely a
cubic map, but an equivariant cubic with a faithful target action.  This is
the input both for the principally polarized case and for cyclic descent
in type $(1,3)$.

\begin{proposition}\label{prop:good-theta}
Let $(B,\Theta)$ be an indecomposable principally polarized complex
abelian surface.  Let $C\subset B$ be a smooth symmetric theta divisor,
and let $0\neq t\in B[3]$ satisfy $t\notin[2]_B(C)$.  Then there is a
dominant rational map $f_t:B\dashrightarrow\PP^2$
of degree three which is equivariant for the translation
$T_t:B\to B$ and a faithful projective action of order three on $\PP^2$.
\end{proposition}

\begin{proof}
Choose $p\in C\cap(C+t)$ and put
\[
  q=p-t,\qquad r=p+t=q-t,\qquad d=p+q.
\]
Then $p,q\in C$.  Moreover, the element $d$ satisfies
\begin{equation}
  d\notin\{0,t,-t\},
\label{eq:d-not}
\end{equation}
because these three equalities would respectively give
$t=2p$, $t=2(-p)$, or $t=2q$, contradicting
$t\notin[2]_B(C)$.

With indices in $\Z/3\Z$, define
\[
  B_i=C+it,\qquad A_i=C+d+it.
\]
Since $C$ represents a principal polarization,
$\operatorname{Stab}(C)\subseteq K(\cO_B(C))=\{0_B\}$.
It follows that the three $B_i$ are pairwise distinct, as are the three
$A_i$.  If $A_i=B_j$, then
$d+(i-j)t\in\operatorname{Stab}(C)$, and hence
$d\in\langle t\rangle=\{0,t,-t\}$, contradicting
\eqref{eq:d-not}.  Thus all six curves are pairwise distinct.
Let $b_i$ and $a_i$ be canonical sections cutting out $B_i$ and $A_i$.
Put $L_i=\cO_B(B_i)$ and
$Q=\phi_{\cO_B(C)}(d)\in\Pic^0(B)$.  By
\eqref{eq:translate-convention} and the theorem of the square,
$\cO_B(A_i)\simeq L_i\otimes Q^{-1}$.
Consequently, the products
\[
  s_i:=a_i b_{i+1}b_{i+2}\qquad(i\in\Z/3\Z)
\]
are sections of the same line bundle
\[
  M=(L_0\otimes L_1\otimes L_2)\otimes Q^{-1}.
\]
Numerically,
\[
  c_1(M)=3[\Theta],\qquad M^2=18.
\]
The sections are linearly independent: restricting a relation to $B_i$
leaves only the restriction of $s_i$, which is not identically zero.

We use the following intersection rule.  If $x,y\in C$ and
$C\neq C+x+y$, then, scheme-theoretically,
\begin{equation}
 C\cap(C+x+y)=
 \begin{cases}
   \{x,y\},&x\neq y,\\
   2[x],&x=y.
 \end{cases}
\label{eq:theta-intersection}
\end{equation}
Indeed, $C=-C$, so the displayed points lie in the intersection, whose
total length is $C^2=2$.  In the double case the involution
$z\mapsto2x-z$ identifies the two tangent lines at $x$.
Applying \eqref{eq:theta-intersection} gives
\begin{equation}
\begin{array}{c|ccc}
 &B_0&B_1&B_2\\ \hline
A_0&\{p,q\}&2[p]&2[q]\\
A_1&2[p]&\{p,r\}&2[r]\\
A_2&2[q]&2[r]&\{q,r\}
\end{array}
\label{eq:A-B-table}
\end{equation}
and
\begin{equation}
\begin{aligned}
 B_0\cap B_1&=\{p,-q\},\\
 B_0\cap B_2&=\{q,-p\},\\
 B_1\cap B_2&=\{r,-r\}.
\end{aligned}
\label{eq:B-B-table}
\end{equation}
The six points
\begin{equation}
  p,q,r,-p,-q,-r
\label{eq:six-theta-points}
\end{equation}
are pairwise distinct.  Indeed, the positive points form the
$\langle t\rangle$-orbit of $p$; if this orbit met its negative, then
$2p\in\langle t\rangle$, and one of $p,-p,p-t\in C$ would have double
$t$, again contradicting goodness.

Equations \eqref{eq:A-B-table}--\eqref{eq:B-B-table} show that
\eqref{eq:six-theta-points} is the complete base support of the linear system.
A point on at least two $B_i$ is a base point.  A point on exactly one
$B_i$ is a base point precisely when it also lies on the corresponding
$A_i$.  A point on none of the $B_i$ would lie on all three $A_i$; its
three-point $t$-orbit would then lie in $A_0\cap A_1$, contradicting
$A_0\cdot A_1=2$.

At each of $p,q,r$, the exact base ideal is $\mathfrak m^2$.  For example,
at $p$, take regular parameters $x,y$ cutting out $B_0,B_1$.  Up to
units,
$(s_0,s_1,s_2)\equiv(y^2,x^2,xy)\pmod{\mathfrak m_p^3}$.
Thus the base ideal $I_p$ satisfies
$I_p+\mathfrak m_p^3=\mathfrak m_p^2$ and
$I_p\subseteq\mathfrak m_p^2$; Nakayama's lemma applied to
$\mathfrak m_p^2/I_p$ gives $I_p=\mathfrak m_p^2$.
At each of $-p,-q,-r$, exactly two $B_i$ meet transversely and no $A_i$
passes through the point, so the exact base ideal is $\mathfrak m$.

Let $\mu_t:X_t\to B$ be the blow-up of the six points, and write $E_z$
for the exceptional curve over $z$.  The moving divisor is
$H_t=\mu_t^*M-2(E_p+E_q+E_r)
-(E_{-p}+E_{-q}+E_{-r})$.  It is base-point-free and
$H_t^2=18-3\cdot2^2-3\cdot1^2=3$.
The induced morphism $X_t\to\PP^2$ is dominant, since a globally generated
divisor defining a map to a curve has square zero, and its degree is
$H_t^2=3$.

Finally, set $D_i=A_i+B_{i+1}+B_{i+2}=\divisor(s_i)$.
Pullback by $T_t$ gives $T_t^*D_i=D_{i-1}$.  Hence translation by $t$
cyclically permutes the three coordinate lines of the linear system, up to nonzero
scalars.  The resulting weighted three-cycle has exact order three in
$\operatorname{PGL}_3(\C)$, so the target action is faithful.
\end{proof}

\section{Principal polarizations}\label{sec:11}

\subsection{The principal-polarization dichotomy}

The following dichotomy is well known.

\begin{proposition}\label{prop:ppas-dichotomy}
Let $(A,\Theta)$ be a principally polarized complex abelian surface.
Exactly one of the following holds.
\begin{enumerate}[(a)]
\item There are elliptic curves $E,E'$ and an isomorphism
      $(E\times E',E\times\{0_{E'}\}+\{0_E\}\times E')
      \simeq(A,\Theta)$ of polarized abelian varieties.
\item The polarization is indecomposable, and $(A,\Theta)$ is the
      canonically principally polarized Jacobian of a smooth genus-two
      curve.  In this case $\Theta$ has a smooth symmetric representative
      $C\subset A$ containing $0_A$.
\end{enumerate}
\end{proposition}

\begin{proof}
This is the standard classification of principally polarized abelian
surfaces \cite[Corollary~11.8.2(a)]{BL}.  In the indecomposable case, choosing
a Weierstrass point as base point for the Abel--Jacobi embedding gives a
smooth symmetric theta divisor containing the origin.
\end{proof}

\subsection{Existence of a good three-torsion point}

\begin{lemma}\label{lem:good-point}
Let $(A,\Theta)$ be an indecomposable principally polarized abelian
surface and let $C\subset A$ be the symmetric genus-two theta divisor from
Proposition~\ref{prop:ppas-dichotomy}.  There exists
$0\neq t\in A[3]$ with $t\notin[2]_A(C)$.
\end{lemma}

\begin{proof}
Let $\Gamma=[2]_A(C)$ with its reduced curve structure.  The restriction
$[2]_A|_C:C\to\Gamma$ is birational.  Indeed, if it had generic degree
greater than one, then for general $x\in C$ there would be
$0\neq\eta\in A[2]$ such that $x+\eta\in C$.  Since $A[2]$ is finite,
one fixed nonzero $\eta$ would satisfy $C+\eta=C$.  This is impossible
because the translation stabilizer of $C$ is contained in
$K(\cO_A(C))=\{0_A\}$.

The projection formula for
$[2]_A:A\to A$ gives
\begin{equation}
  \Gamma\equiv4\Theta,\qquad \Gamma^2=32.
\label{eq:Gamma-square}
\end{equation}
Suppose that every nonzero point of $A[3]$ belongs to $\Gamma$.  Fix
$0\neq a\in A[3]$.  The curves $\Gamma$ and $\Gamma+a$ are distinct.
Otherwise $T_a:\Gamma\to\Gamma$ would lift through the normalization
$C\to\Gamma$ to a fixed-point-free automorphism of order three of the
genus-two curve $C$, contradicting Riemann--Hurwitz.

For every $x\in A[3]\setminus\{0_A,a\}$, both $x$ and $x-a$ are
nonzero, so
$x\in\Gamma\cap(\Gamma+a)$.  This gives at least $79$ distinct
intersection points, contradicting
\eqref{eq:Gamma-square}.  Hence a good $t$ exists.
\end{proof}

\begin{theorem}\label{thm:type11}
If a complex abelian surface $A$ admits a principal polarization, then
$\irr(A)=3$.
\end{theorem}

\begin{proof}
If the principal polarization is decomposable,
Proposition~\ref{prop:ppas-dichotomy} identifies $A$ with a product of
elliptic curves, we conclude by Proposition~\ref{prop:product}.  If it is indecomposable,  we conclude by Lemma~\ref{lem:good-point}
and Proposition~\ref{prop:good-theta}.
\end{proof}

\section{Polarizations of type \texorpdfstring{$(1,2)$}{(1,2)}}
\label{sec:12}

Let $L$ be a line bundle of type $(1,2)$ on an abelian surface $A$.
Then $L^2=4$ and $h^0(A,L)=2$, so $|L|$ is a pencil.
The fixed-part-free geometry was already described by Barth
\cite{Barth}; the formulation below, including
the product case, follows from Lemma~\ref{lem:BL-1d}.

\begin{proposition}\label{prop:12-dichotomy}
Exactly one of the following holds.
\begin{enumerate}[(a)]
\item The pencil $|L|$ has no fixed divisor.  Its base scheme is a reduced
      four-point torsor under $K(L)\simeq(\Z/2\Z)^2$,
      and a general member of $|L|$ is a smooth irreducible curve of genus
      three.
\item There are elliptic curves $E_1,E_2$, line bundles $P_i$ on $E_i$,
      and an origin-preserving isomorphism
      $\alpha:E_1\times E_2\xrightarrow{\sim}A$ such that
      $\deg(P_1)=1$, $\deg(P_2)=2$, and
      $\alpha^*L\simeq p_1^*P_1\otimes p_2^*P_2$.
      In this case $|L|$ has a fixed elliptic component.
\end{enumerate}
\end{proposition}

\begin{proof}
The alternatives and the product description are exactly
Lemma~\ref{lem:BL-1d}(i).  In the absence of a fixed component,
Lemma~\ref{lem:BL-1d}(ii) gives a four-point base scheme and a smooth
general member $C\in|L|$.  The base scheme is invariant under the free
translation action of $K(L)$, which has order four.  Since its length is
$L^2=4$, it is a single reduced $K(L)$-orbit.  The ample divisor $C$ is
connected and, being smooth, is therefore irreducible.  Adjunction gives
$2g(C)-2=C^2=4$, so $g(C)=3$.
\end{proof}

\begin{theorem}\label{thm:type12}
If a complex abelian surface $A$ admits a polarization of type $(1,2)$,
then $\irr(A)=3$.
\end{theorem}

\begin{proof}
In case~(a) of Proposition~\ref{prop:12-dichotomy}, the surface $A$
contains a smooth curve of genus three.  Tokunaga--Yoshihara shows that precisely
this hypothesis supplies a dominant rational map
$A\dashrightarrow\PP^2$ of degree three
\cite{TokunagaYoshihara, YoshiharaProduct}.  In case~(b),
$A$ is a product of elliptic curves and Proposition~\ref{prop:product}
applies.  
\end{proof}

\section{Polarizations of type \texorpdfstring{$(1,3)$}{(1,3)}}
\label{sec:13}

\subsection{Good three-torsion and the good locus}

Recall from Definition~\ref{def:good-torsion} that, for a smooth
symmetric theta divisor $C$, a nonzero point $t\in B[3]$ is good when
$t\notin[2]_B(C)$ and bad otherwise.  Proposition~\ref{prop:good-theta}
handles good points.  We now deform a prescribed bad point to the good
locus and retain the equivariant linear system under specialization.

We use the fine level-$3$ moduli space with a universal family $\pi: \mathscr{B}\to \mathcal{A}_2(3)$ and its finite \'etale
theta-line-bundle cover described in
Subsection~\ref{subsec:level-structures}.  Fix a primitive cube root
$\zeta$, put $W=(\mathbf F_3)^4$ with its standard symplectic form, and
fix $0\ne v\in W$.  The universal marking is denoted
$\alpha:W\xrightarrow{\sim}\mathscr B[3]$, the induced torsion section is
$t=\alpha(v)$, and $\mathscr C\subset\mathscr B$ denotes the relative
symmetric theta divisor.  Changing the marking only relabels torsion
points; it does not translate the theta divisor.

\begin{lemma}\label{lem:13-good-locus}
On any connected component $S$ of the finite \'etale
theta-line-bundle cover above, let $t=\alpha(v)$ be the universal
three-torsion section.  For a point $s\in S$, write
$\bar s=\Spec\overline{\kappa(s)}$ for a geometric point over $s$, and
define
\[
S^{\mathrm{good}}
:=
\left\{
s\in S:
t_{\bar s}\notin [2](\mathscr C_{\bar s})
\right\}.
\]
Then $S^{\mathrm{good}}$ is a nonempty Zariski-open subset.  Moreover,
if
\[
S^{\mathrm{ind}}
:=
\left\{
s\in S:
(\mathscr B_{\bar s},\mathscr C_{\bar s})
\text{ is indecomposable}
\right\},
\]
then $S^{\mathrm{good}}\cap S^{\mathrm{ind}}$ is nonempty.
\end{lemma}

\begin{proof}
The scheme $\mathcal A_2(3)$ is smooth and irreducible.  Thus its finite
\'etale cover is normal, and the connected component $S$ is irreducible
(the irreducible components of a normal scheme are open and closed).

Consider the closed subscheme
\[
Z_{\mathrm{bad}}
:=
\mathscr C\cap [2]_{\mathscr B}^{-1}(t(S))
\subset \mathscr B.
\]
Since $[2]_{\mathscr B}^{-1}(t(S))$ is finite \'etale over $S$,
the morphism $Z_{\mathrm{bad}}\to S$ is finite.  Hence its image
$D_{\mathrm{bad}}\subset S$ is closed.

For every $s\in S$, after passing to a geometric point
$\bar s=\Spec\overline{\kappa(s)}$ over $s$, we have
\[
s\in D_{\mathrm{bad}}
\iff
(Z_{\mathrm{bad}})_{\bar s}\neq\varnothing
\iff
\exists\,x\in\mathscr C_{\bar s}
\text{ with }[2]x=t_{\bar s}
\iff
t_{\bar s}\in[2](\mathscr C_{\bar s}).
\]
Thus
\[
D_{\mathrm{bad}}=S\setminus S^{\mathrm{good}},
\]
and $S^{\mathrm{good}}$ is Zariski open.  Notice in particular that
this description is geometric and is preserved by extension of the
residue field.

It remains to prove nonemptiness.  Take a product principally polarized
surface
\[
B=E\times F
\]
with product theta divisor
\[
C_0
=
E\times\{0_F\}\cup\{0_E\}\times F.
\]
Then $[2](C_0)=C_0$.  Choose nonzero
$\alpha_0\in E[3]$ and $\beta_0\in F[3]$.  The point
$(\alpha_0,\beta_0)$ does not lie in $[2](C_0)$.
By the transitivity discussed above, one may choose the level marking
so that the fixed abstract vector $v$ labels this point.  This operation
leaves $C_0$ unchanged.

Moreover, every symmetric theta divisor inducing the product principal
polarization is a translate of $C_0$ by a two-torsion point, and hence
has the same image under $[2]$.  Since every connected component of a
finite \'etale cover of the connected normal scheme
$\mathcal A_2(3)$ maps surjectively to $\mathcal A_2(3)$, this good
marked product point has a lift to every connected component $S$.
Consequently $D_{\mathrm{bad}}\neq S$.

The geometric decomposable locus is a proper closed subset of
$\mathcal A_2(3)$; hence its inverse image
$D_{\mathrm{dec}}\subset S$ is also a proper closed subset.  Since
$S$ is irreducible,
\[
D_{\mathrm{bad}}\cup D_{\mathrm{dec}}\subsetneq S.
\]
Equivalently,
\[
S^{\mathrm{good}}\cap S^{\mathrm{ind}}\neq\varnothing.
\]
\end{proof}

\begin{lemma}\label{lem:13-good-deformation}
Let $(B,C)$ be an indecomposable principally polarized complex
abelian surface with a chosen smooth symmetric theta divisor, and let
$0\ne t\in B[3]$.  There exist a DVR $R$, with residue field
$k=\C$ and fraction field $K$, a principally polarized abelian scheme
\[
\pi:\mathscr B\to\Spec R,
\]
a three-torsion section still denoted $t$, and a symmetric rigidified
relative theta line bundle
\[
\mathscr L=\mathcal O_{\mathscr B}(\mathscr C)
\]
with relative theta divisor $\mathscr C\subset\mathscr B$, such that:

\begin{enumerate}[(i)]
\item the special fiber is the prescribed marked object $(B,C,t)$;
\item for an algebraic closure $\bar K/K$, the geometric generic fiber
$(\mathscr B_{\bar K},\mathscr C_{\bar K})$ is indecomposable and
$t_{\bar K}$ is good.
\end{enumerate}
\end{lemma}

\begin{proof}
Choose a full symplectic level-$3$ marking of $B[3]$ carrying the fixed
vector $v$ to $t$, and lift the resulting point to the
theta-line-bundle cover using the symmetric line bundle
$\mathcal O_B(C)$.  Let $S$ be the connected component containing this
point.  By Lemma~\ref{lem:13-good-locus}, the intersection
\[
U
:=
S^{\mathrm{good}}\cap S^{\mathrm{ind}}
\]
is a nonempty Zariski-open subset of $S$.

Since $S$ is smooth and irreducible, general hyperplane sections through
the prescribed point give an integral curve $T\subset S$ which is not
contained in $S\setminus U$.  Since $S\setminus U$ is
closed and $T$ is integral, the generic point $\eta_T$ of $T$ belongs
to $U$.  Let
\[
\nu:\widetilde T\to T
\]
be the normalization, choose a point $\tilde x\in\widetilde T$ lying
over the prescribed point, and set
\[
R:=\mathcal O_{\widetilde T,\tilde x}.
\]
Then $R$ is a DVR.  Since the residue field of $\tilde x$ is a finite
extension of $\C$, it is equal to $\C$.  Let $K=\operatorname{Frac}(R)$.

Pulling back the universal objects along
$\Spec R\to S$ gives the asserted family.  Over the local base, tensoring
the theta line bundle by the pullback of the inverse of its restriction
to the zero section supplies the stated rigidification.  Its closed point maps to the
prescribed marked surface.  Its generic point
\[
\Spec K\longrightarrow S
\]
maps to $\eta_T\in U$.  Therefore, after base change to an algebraic
closure $\bar K/K$, the geometric generic point
\[
\Spec\bar K\longrightarrow\Spec K\longrightarrow S
\]
still factors through $U$.  By the geometric description of $U$ in
Lemma~\ref{lem:13-good-locus}, it follows that
$(\mathscr B_{\bar K},\mathscr C_{\bar K})$ is indecomposable and
$t_{\bar K}$ is good.  The required smoothness is geometric.  Indeed,
over an algebraically closed field of characteristic zero an
indecomposable theta divisor on a principally polarized abelian surface
is integral and has arithmetic genus two.  If it were singular, its
normalization would have genus at most one.  A rational curve cannot map
nontrivially to an abelian variety, whereas the image of an elliptic curve
is a translate of an elliptic subgroup and has self-intersection zero,
contrary to $\mathscr C_{\bar K}^2=2$.  Thus
$\mathscr C_{\bar K}$ is a smooth curve of genus two.
\end{proof}

\subsection{Specialization of a dominant linear system}

Chen and Stapleton prove a general specialization theorem for rational
maps to ruled varieties \cite[Theorem~1.1]{ChenStapleton}.  That theorem
does not retain a prescribed target or an equivariant linear system, so
we use the following fixed-linear-system refinement.  Its explicit
base-change hypothesis is essential both in the proof and below.

\begin{lemma}\label{lem:degree-specialization}
Let $R$ be a DVR with fraction field $K$ and residue field $k$.  Let
$\pi:\mathscr X\to\Spec R$ be a smooth projective family of geometrically
integral surfaces, let $\mathscr M$ be a line bundle, and suppose that
$\mathscr E=\pi_*\mathscr M$ is finite free and commutes with base change.
Let $\mathscr V\subset\mathscr E$ be a saturated free submodule of rank
three.  Put $V_K=\mathscr V\otimes_RK$ and
$V_0=\mathscr V\otimes_Rk$, and assume that their evaluation maps define
dominant rational maps $f_K:\mathscr X_K\dashrightarrow\PP(V_K^\vee)$
and $f_0:\mathscr X_0\dashrightarrow\PP(V_0^\vee)$.  Then
$\deg(f_0)\leq\deg(f_K)$.
\end{lemma}

\begin{proof}
Because $\mathscr E/\mathscr V$ is torsion-free over the DVR, it is flat.
Tensoring $0\to\mathscr V\to\mathscr E\to
\mathscr E/\mathscr V\to0$ with $k$ therefore shows that
$V_0\hookrightarrow\mathscr E\otimes_Rk$.  The base-change assumption
identifies the latter module with $H^0(\mathscr X_0,\mathscr M_0)$, so
$V_0$ induces a $2$-dimensional linear system on $\mathscr{X}_0$.

Consider the relative evaluation map
\[
 \operatorname{ev}:\mathscr V\otimes_R\mathcal O_{\mathscr X}
 \longrightarrow\mathscr M
\]
and let $U\subset\mathscr X$ be its surjectivity locus.  Nakayama's
lemma identifies $U_0$ with the base-point-free locus of $V_0$, and
evaluation induces a morphism
$F:U\to\PP(\mathscr V^\vee)$ whose restrictions represent the rational
maps in the statement.  Since $U$ is open in the $R$-flat scheme
$\mathscr X$, its generic fiber is schematically dense in $U$.
Consequently, the graph of $F|_{U_0}$ is contained in the special fiber
of the closure of the generic graph.  This inclusion is the reason for
the base-change and saturation hypotheses.

Let $\mathscr G$ be the reduced closure of the graph of $f_K$ in
$\mathscr X\times_R\PP(\mathscr V^\vee)$.  It is integral.  Its local
rings are torsion-free over $R$, hence $\mathscr G$ is flat over the DVR.
Every component of $\mathscr G_0$ has dimension two, because a
uniformizer is a nonzerodivisor on the integral threefold $\mathscr G$.
The closure $G_0$ of the graph of the dominant map $f_0$ also has
dimension two.  The preceding graph inclusion shows that
$G_0$ is an irreducible component of the fundamental cycle of
$\mathscr G_0$, with multiplicity $m\geq1$.  For
$\xi=\operatorname{pr}_2^*\mathcal O(1)|_{\mathscr G}$, specialization of
Cartier intersections \cite[\S20.3]{Fulton} gives
\begin{equation}
 \deg\bigl(\xi^2\cap[\mathscr G_0]\bigr)
 =\deg\bigl(\xi^2\cap[\mathscr G_K]\bigr)=\deg(f_K).
\label{eq:graph-specialization}
\end{equation}
The line bundle $\xi$ is nef, so its square has nonnegative degree on
every other two-dimensional component of $\mathscr G_0$.  Its degree on
$G_0$ is $\deg(f_0)$.  Equation~\eqref{eq:graph-specialization} therefore
gives $m\deg(f_0)\leq\deg(f_K)$ and hence the desired inequality.
\end{proof}

\subsection{Degeneration of the good-torsion construction}

Assume from now on that the special point $t\in B[3]$ is bad, and take
the family of Lemma~\ref{lem:13-good-deformation}.  The two relative
divisors $\mathscr C$ and $\mathscr C+t$ are distinct on each geometric
fiber.  On the special and
geometric generic fibers they are integral, hence have no common
component.  Their intersection
$Z:=\mathscr C\cap(\mathscr C+t)$ is proper and quasi-finite over $R$,
hence finite.  It is a codimension-two local complete intersection in
the smooth threefold $\mathscr B$, so it is Cohen--Macaulay and pure of
dimension one.  Its fibers are zero-dimensional; consequently it has
no vertical associated component, and a uniformizer of $R$ is a
nonzerodivisor on $Z$.  Thus $Z$ is finite flat over $R$, of fiber
length $\Theta^2=2$.

Choose a closed point of $Z_{\bar K}$, defined over a finite extension
$K'/K$.  Normalize $R$ in $K'$, localize at a point over the closed
point of $\Spec R$, and use properness to extend the resulting generic
point to a section $p$ of $Z$.  This finite extension may be ramified;
after replacing $R$ by it, the residue field is still $\C$, and goodness
and indecomposability of the geometric generic fiber are unchanged.

Put $q:=p-t$ and $d:=p+q$.  With indices in $\Z/3\Z$, set
$\mathscr B_i:=\mathscr C+it$ and
$\mathscr A_i:=\mathscr C+d+it$, and let $b_i$ and $a_i$ be their
canonical relative sections.  Write
$\mathscr L_i:=\mathcal O_{\mathscr B}(\mathscr B_i)$ and
$\mathscr Q:=\phi_{\mathscr L}(d)$.  The theorem of the square, with the
translation convention~\eqref{eq:translate-convention}, identifies the
three products
\[
 s_i=a_i b_{i+1}b_{i+2}\quad(i\in\Z/3\Z)
 \quad\text{as sections of}\quad
 \mathscr M:=(\mathscr L_0\otimes\mathscr L_1\otimes\mathscr L_2)
 \otimes\mathscr Q^{-1}.
\]
The relative theorem of the square in the rigidified Picard functor
gives these identifications on the abelian scheme over $R$, rather than
only on its generic fiber.  Each canonical divisor section is primitive,
because its relative Cartier divisor contains no entire fiber, and a
line-bundle isomorphism over $R$ cannot introduce a factor of a
uniformizer.  Hence the product sections $s_i$ are primitive integral
sections and their reductions are nonzero.  For every geometric point
$u\in\Spec R$, the subscript $u$ will denote restriction to that fiber.
Every $\mathscr M_u$ is numerically equivalent to $3\Theta_u$, so its
higher cohomology vanishes and
$h^0(\mathscr B_u,\mathscr M_u)=9$.  Thus cohomology and base change show
that
$\mathscr E=\pi_*\mathscr M$ is free of rank nine and commutes with base
change.

Translation by $t$ preserves $\mathscr M$.  Choose an isomorphism
$T_t^*\mathscr M\simeq\mathscr M$.  Since
$\pi_*\mathcal O_{\mathscr B}=\mathcal O_R$, its third iterate is
multiplication by a unit $u\in R^*$.  The algebra
$R[X]/(X^3-u)$ is finite \'etale.  After making this base change,
localizing over the closed point, and rescaling by the inverse of the
resulting cube root of $u$, we may arrange that the third iterate is the
identity.  Thus translation by $t$ gives a
$G=\Z/3\Z$-action on $\mathscr E$.  The generic span
$V_K=\langle s_0,s_1,s_2\rangle_K\subset\mathscr E_K$ is $G$-stable.
Define its saturated lattice by
$\mathscr V=V_K\cap\mathscr E\subset\mathscr E_K$ and put
$V_0=\mathscr V\otimes_Rk$.  The intersection definition makes
$\mathscr E/\mathscr V$ torsion-free, so $\mathscr V$ is a saturated free
rank-three module and $V_0\hookrightarrow H^0(B,\mathscr M_0)$.

\begin{lemma}
\label{lem:13-limit-system}
The action of $G$ on $V_0$ contains each of the three characters of $G$
with multiplicity one.  In particular its action on
$\PP(V_0^\vee)$ is faithful.  Moreover, the linear system $V_0$ has no
fixed divisor.
\end{lemma}

\begin{proof}
Let $\sigma$ generate $G$.  On the generic basis
$(s_0,s_1,s_2)$, the operator $\sigma$ is a weighted three-cycle.  The
linearization was normalized so that $\sigma^3=1$, and therefore its
three eigenvalues are $1,\zeta,\zeta^2$, each once.  Since
$R$ contains $1/3$ and $\zeta$, the idempotents
$e_j=\frac13\sum_{m=0}^2\zeta^{-jm}\sigma^m$ split
$\mathscr V$ into three direct summands.  Each summand has generic rank
one and hence, over the local ring $R$, is free of rank one.  Reduction
therefore gives the three distinct one-dimensional character spaces in
$V_0$.  Thus a generator acts, up to a scalar and a choice of basis, as
$\operatorname{diag}(1,\zeta,\zeta^2)$, which has exact order three in
$\operatorname{PGL}(V_0)$.

Choose $x\in C$ with $2x=t$.
The involution $\iota_t(z)=t-z$ sends $C$ to $C+t$, fixes $x$, and has
differential $-1$ at $x$; it therefore identifies the tangent lines of
$C$ and $C+t$ there.  Since the two distinct theta divisors have
intersection number two, their scheme-theoretic intersection is
$C\cap(C+t)=2[x]$.
It follows that on
the special fiber $p_0=x$, $q_0=-x$, and $d_0=0$.  Hence, on the special
fiber, we have
$\mathscr A_{i,0}=\mathscr B_{i,0}=B_i=C+it$, and the reduction of each
$s_i$ is a nonzero scalar multiple of
\begin{equation}
 P=b_0b_1b_2\in V_0,\qquad \operatorname{div}(P)=B_0+B_1+B_2.
\label{eq:limit-product}
\end{equation}
Indeed, every $s_i$ lies in
$V_K\cap\mathscr E=\mathscr V$, and its primitive nonzero reduction is
a scalar multiple of $P$.  The injection
$V_0\hookrightarrow\mathscr E\otimes_Rk$ therefore proves that
$0\ne P\in V_0$.

Let $F$ be the maximal fixed divisor of $V_0$.  Since $P\in V_0$,
equation~\eqref{eq:limit-product} gives $F\leq B_0+B_1+B_2$.  The three
curves $B_i$ are distinct and irreducible, and $G$ cyclically permutes
them.  Since $V_0$ is $G$-stable, $F$ is $G$-invariant; hence either
$F=0$ or $F=B_0+B_1+B_2$.  In the second case every section in $V_0$
would be $P$ times a global regular function on $B$.  Since
$H^0(B,\mathcal O_B)=\C$, this would give $\dim V_0\leq1$, a
contradiction.  Thus $F=0$.
\end{proof}

\begin{lemma}\label{lem:13-limit-dominant}
The rational map $f_0:B\dashrightarrow\PP(V_0^\vee)$ is dominant.
\end{lemma}

\begin{proof}
The image cannot be a point because $V_0$ has dimension three.  Suppose
that it is an integral plane curve $\Gamma$.  The curve is nondegenerate and thus its degree
$e$ satisfies $e\geq2$.

Resolve the rational map by $\mu:X'\to B$, factor the induced morphism
through the normalization $\widetilde\Gamma$, and take Stein
factorization
$X'\xrightarrow{h}T\xrightarrow{\eta}\widetilde\Gamma$, where $T$ is a
normal projective curve, $h$ has connected fibers, and
$\delta=\deg(\eta)$.  For a general $y\in T$, let $F_y=h^{-1}(y)$ and
let $D$ denote the numerical class of $\mu_*F_y$ on $B$.  Because
$V_0$ has no fixed divisor, pushing a general hyperplane section from
$X'$ to $B$ gives
\begin{equation}
 3[C]=e\delta[D]\qquad\text{in }\operatorname{NS}(B).
\label{eq:theta-fiber-class}
\end{equation}
Indeed, $\mathcal O_{\widetilde\Gamma}(1)$ has degree $e$, its pullback
to $T$ has degree $e\delta$, and a general divisor in that pullback is a
sum of $e\delta$ general fibers of $h$; exceptional divisors disappear
under $\mu_*$.  The pushforward of a general member of the moving linear
system has class $c_1(\mathscr M_0)=3[C]$.

Since $[C]$ is primitive and
$\operatorname{NS}(B)$ is free abelian, there is a homomorphism
$\ell:\operatorname{NS}(B)\to\Z$ with
$\ell([\Theta])=1$.  Applying it to
\eqref{eq:theta-fiber-class} gives
$3=e\delta\,\ell(D)$; hence $e\delta$ divides $3$.  Together with
$e\geq2$ and $\delta\geq1$, this forces
$e=3$, $\delta=1$, and then
\eqref{eq:theta-fiber-class} gives $[D]=[\Theta]$.

An integral plane cubic has normalization of genus zero or one.  If
$T\simeq\widetilde\Gamma\simeq\PP^1$, two general fibers of $h$ are
linearly equivalent, and pushing the divisor of the corresponding
rational function through the birational morphism $\mu$ shows that the
distinct divisors $D_y=\mu_*F_y$ are linearly equivalent on $B$.  Thus
the line bundle $N=\mathcal O_B(D_y)$ has at least two independent
sections.  But
$c_1(N)=[C]$ and $h^0(B,N)=1$, a contradiction.

If $T\simeq\widetilde\Gamma$ is elliptic, the induced rational map
$B\dashrightarrow T$ extends to a morphism.  Its two
general fibers are disjoint, so $D^2=0$, contradicting
$[D]=[C]$.  Both possibilities are impossible,
and $f_0$ is dominant.
\end{proof}

\begin{proposition}\label{prop:all-t}
Let $(B,\Theta)$ be an indecomposable principally polarized complex
abelian surface, let $C\subset B$ be a smooth symmetric theta divisor,
and let $0\ne t\in B[3]$.  There is a $T_t$-equivariant dominant rational
map $B\dashrightarrow\PP^2$ of degree three for a faithful projective
order-three action on $\PP^2$.
\end{proposition}

\begin{proof}
If $t$ is good, this is Proposition~\ref{prop:good-theta}.  If $t$ is
not good, use the one-parameter construction above.  By
Lemmas~\ref{lem:13-limit-system} and~\ref{lem:13-limit-dominant}, the
specialized space $V_0$ defines a dominant equivariant rational map to
$\PP(V_0^\vee)$, and the target action is faithful.  The proof of
Proposition~\ref{prop:good-theta} is algebraic and remains valid over
every algebraically closed field of characteristic zero.  Applied to the
geometric generic fiber, it shows that $(f_K)_{\bar K}$ is dominant of
degree three.  Since source and target are geometrically integral,
dominance and degree descend to $K$.  Lemma~\ref{lem:degree-specialization} gives
$\deg(f_0)\leq3$, while the Alzati--Pirola lower bound recalled in the
Introduction gives $\deg(f_0)\geq3$.  Hence we obtain
$\deg(f_0)=3$.
\end{proof}

\subsection{A cyclic principally polarized cover}

\begin{lemma}\label{lem:13-cover}
Let $(A,L)$ be a complex abelian surface polarized of type $(1,3)$.
There are a principally polarized abelian surface $(B,\Theta)$, an \'etale
cyclic isogeny $q:B\to A$ of degree three, and a nonzero $t\in B[3]$
such that $\ker(q)=\langle t\rangle$ and
$q^*c_1(L)=3c_1(\Theta)$.
\end{lemma}

\begin{proof}
Write $A=V/\Lambda$ and choose a symplectic basis
$e_1,e_2,f_1,f_2$ for the integral Riemann form $E_L$, with
$E_L(e_1,f_1)=1$ and $E_L(e_2,f_2)=3$.  Set
$\Lambda'=\Z(3e_1)\oplus\Z e_2\oplus
\Z f_1\oplus\Z f_2$ and $B=V/\Lambda'$.  The identity on
$V$ induces an \'etale cyclic isogeny $q:B\to A$ of degree three, with
kernel generated by $t=e_1+\Lambda'$.  The form
$E_\Theta=\frac13E_L|_{\Lambda'}$ is integral, positive, and unimodular.
It defines a principal polarization and the equality of Riemann forms is
exactly $q^*c_1(L)=3c_1(\Theta)$.
\end{proof}

\subsection{The decomposable covering surface}

\begin{lemma}\label{lem:equivariant-product}
Let $B=E\times F$ be a product of elliptic curves and let
$t=(\alpha,\beta)\in B[3]$ with $\alpha,\beta\ne0$.  There is a
$T_t$-equivariant dominant rational map $B\dashrightarrow\PP^2$ of
degree three, for a faithful projective order-three action on $\PP^2$.
\end{lemma}

\begin{proof}
Choose $x\in E$ and $z\in F$ with $3x=-\alpha$ and $3z=\beta$; both
have exact order nine.  Put $x_i=x+i\alpha$ and $z_i=z-i\beta$.  In the
cubic embedding of $E$ by $|3[0_E]|$, the tangent at $x_i$ meets $E$
again at $-2x_i=x_{i+1}$.  The three points $x_i$ are noncollinear,
because $x_0+x_1+x_2=3x=-\alpha\ne0$.  The three tangents in question
are therefore the three sides of this triangle, and their equations form
a basis.  Choose the resulting coordinate sections
$a_i\in H^0(E,\mathcal O_E(3[0_E]))$ with
$\operatorname{div}(a_i)=2x_{i+1}+x_{i+2}$ and
$T_\alpha^*\operatorname{div}(a_i)= \operatorname{div}(a_{i-1})$.  At $x_0,x_1,x_2$ the order triples of
$(a_0,a_1,a_2)$ are respectively $(0,1,2)$, $(2,0,1)$, and $(1,2,0)$.

Put $y_0=z_0$, $y_1=z_2$, and $y_2=z_1$.  Choose
$n_i\in H^0(F,\mathcal O_F(3[0_F]))$ with
$\operatorname{div}(n_i)=2y_{i-1}+y_{i+1}$ and
$T_\beta^*\operatorname{div}(n_i) = \operatorname{div}(n_{i-1})$.  At $y_0,y_1,y_2$ the order triples of
$(n_0,n_1,n_2)$ are respectively $(0,2,1)$, $(1,0,2)$, and $(2,1,0)$.
Here the corresponding three points are again noncollinear, since their
sum is $3z=\beta\ne0$, so the sections $n_i$ form a basis as well.
The three products $a_i\boxtimes n_i$ define a rational map to $\PP^2$.
Their common line bundle has square eighteen.  The base support consists
of the six points $(x_i,y_j)$ with $i\ne j$; the exact base ideal is
$\mathfrak m^2$ at $(x_0,y_2),(x_1,y_0),(x_2,y_1)$ and is
$\mathfrak m$ at the other three points, as the two displayed order
tables show.  One blow-up therefore resolves the linear system, and the
moving divisor has square $18-3\cdot2^2-3\cdot1^2=3$.  A base-point-free
divisor of positive square cannot define a map to a curve, so the map is
dominant and has degree three.  Finally, translation by $t$ cyclically
permutes the
three product sections up to nonzero scalars, which gives a faithful
projective action of order three on the target.
\end{proof}

\subsection{Descent and the type \texorpdfstring{$(1,3)$}{(1,3)} theorem}

\begin{lemma}\label{lem:cyclic-descent}
Let $G\simeq\Z/3\Z$ act faithfully on a smooth projective
surface $B$ and faithfully projectively on $\PP^2$.  If a dominant
rational map $f:B\dashrightarrow\PP^2$ of degree three is
$G$-equivariant, then the induced rational map
$B/G\dashrightarrow\PP^2/G$ has degree three, and $\PP^2/G$ is rational.
\end{lemma}

\begin{proof}
A finite-order projective transformation is diagonalizable.  After
scaling a lift, a generator acts as
$\operatorname{diag}(1,\zeta^a,\zeta^b)$, where
$a,b\in\{0,1,2\}$ are not both zero.  On the affine chart with coordinates
$(u,v)$, the invariant field is, after possibly interchanging the
coordinates, one of
$\C(u^3,v)$, $\C(u/v,v^3)$, or $\C(u^3,uv)$.  Thus the quotient surface
$\PP^2/G$ is rational.

Let $F=\C(B)$ and let $K=f^*\C(\PP^2)$.  Equivariance makes $K$ a
$G$-stable subfield.  Faithfulness on $B$ and on the target gives
$[F:F^G]=[K:K^G]=3$, while $[F:K]=3$.  The tower law yields
$[F^G:K^G]=[F:K][K:K^G]/[F:F^G]=3$, which is the degree of the descended
map.
\end{proof}

\begin{theorem}\label{thm:type13}
If a complex abelian surface $A$ admits a polarization of type $(1,3)$,
then $\irr(A)=3$.
\end{theorem}

\begin{proof}
Use Lemma~\ref{lem:13-cover} to write $A=B/\langle t\rangle$ with
$(B,\Theta)$ principally polarized.  If $(B,\Theta)$ is indecomposable,
Proposition~\ref{prop:all-t} supplies the required equivariant cubic.  If
it is decomposable, write $B=E\times F$ and $t=(\alpha,\beta)$.  When one
component is zero, $A$ is itself a product of elliptic curves and
Proposition~\ref{prop:product} applies.  When both are nonzero,
Lemma~\ref{lem:equivariant-product} supplies the equivariant cubic on
$B$.  In every equivariant case Lemma~\ref{lem:cyclic-descent} gives a
degree-three rational map from $A$ to the rational surface $\PP^2/G$.
Composing it with a birational map $\PP^2/G\dashrightarrow\PP^2$
preserves the degree.  Hence we obtain $\irr(A)\leq3$, and the
Alzati--Pirola lower bound recalled in the
Introduction gives equality.
\end{proof}

\section{Polarizations of type \texorpdfstring{$(1,6)$}{(1,6)}}
\label{sec:16}

Moretti constructs the degree-three map for a general
$(1,6)$-polarized abelian surface from a rank-two direct image and the
good-pairs correspondence; see \cite[Proposition~1.5 and \S3]{Moretti} or Subsection~\ref{subsec:moretti-good-pairs}.
Under the auxiliary hypotheses imposed there, Grossi--Moretti further
study its fibers and symmetries
\cite[Theorems~2.1--2.2]{GrossiMoretti}.  We use the same construction,
but retain only the unconditional good-pairs
correspondence and treat the possible divisorial failure of global
generation, as well as the case in which the resulting map is composed
with a pencil, separately.

\subsection{Construction of the good pair}

\begin{proposition}\label{prop:16-cover}
Let $(A,L)$ be a complex abelian surface polarized of type $(1,6)$.
There exist a degree-two \emph{\'etale} isogeny
$\pi:B\to A$, with kernel generated by a nonzero point $b\in B[2]$,
and a line bundle $M$ of type $(1,3)$ on $B$ such that
$2c_1(M)=\pi^*c_1(L)$.  If
\[
 E:=\pi_*M,\qquad W:=H^0(A,E)=H^0(B,M),\qquad Q:=\det E,
\]
then
\[
 \operatorname{rk}E=2,\quad \dim W=3,\quad c_2(E)=3,
 \quad [Q]=[L]\text{ in }\operatorname{NS}(A),
 \quad \pi^*E\simeq M\oplus T_b^*M.
\]
\end{proposition}

\begin{proof}
This is the lattice construction underlying the quotient presentation
used in \cite[\S3, first paragraph]{Moretti}.  Concretely, for a symplectic
basis $e_1,e_2,f_1,f_2$ of the Riemann form of $L$, pass to the
index-two sublattice generated by $2e_1,e_2,f_1,f_2$; one half of the
restricted Riemann form has type $(1,3)$.  Finite flat base change gives
\[
 \pi^*\pi_*M\simeq M\oplus T_b^*M.
\]
It follows that $\pi^*c_1(Q)=2c_1(M)=\pi^*c_1(L)$, and the injectivity
of $\pi^*:\operatorname{NS}(A)\to\operatorname{NS}(B)$ gives
$[Q]=[L]$.  Finally,
\[
 2c_2(E)=c_2(\pi^*E)=c_1(M)c_1(T_b^*M)=M^2=6,
 \qquad h^0(B,M)=\frac{M^2}{2}=3.
\]
\end{proof}

\begin{lemma}\label{lem:16-product-cover}
If $B$ in Proposition~\ref{prop:16-cover} is a product of elliptic
curves, then $A$ is a product of elliptic curves or admits a
polarization of type $(1,2)$.  In either case $\irr(A)=3$.
\end{lemma}

\begin{proof}
Write $B=E_1\times E_2$ and $b=(b_1,b_2)$.  If one component of $b$
vanishes, then $B/\langle b\rangle$ is a product.  If both are nonzero,
the images $\overline E_1,\overline E_2\subset A$ of the two factors
meet transversely in the two points
$0_A$ and $\pi(b_1,0)=\pi(0,b_2)$.  Thus
$\overline E_1\cdot\overline E_2=2$, while
$\overline E_1+\overline E_2$ is ample.  Its square is four, so it has
type $(1,2)$.  Apply Proposition~\ref{prop:product} or
Theorem~\ref{thm:type12}, respectively.
\end{proof}

Assume from now on that $B$ is not a product of elliptic curves.  By
Lemma~\ref{lem:BL-1d}, $|M|$ is base-point-free and a general member is
smooth; it is connected because $M$ is ample, and hence is integral.

\begin{proposition}\label{prop:16-good-pair}
Let
\[
 \operatorname{ev}:W\otimes\mathcal O_A\longrightarrow E,
 \qquad
 \mathcal F:=\bigl(\operatorname{im}(\operatorname{ev})\bigr)^{\vee\vee},
 \qquad \mathcal P:=\det\mathcal F.
\]
Then $(\mathcal F,W)$ is a good pair in the sense of Moretti.  Its
determinant system defines a nondegenerate rational map
\[
 \Phi:A\dashrightarrow\PP(W)\simeq\PP^2.
\]
If $\Phi$ is dominant, then $\deg\Phi=3$.
\end{proposition}

\begin{proof}
For $s\in W=H^0(B,M)$ put $D_s=Z_B(s)$.  Under the splitting of
Proposition~\ref{prop:16-cover},
\begin{equation}\label{eq:16-zero-pullback}
 \pi^{-1}\bigl(Z_E(s)\bigr)=D_s\cap T_b(D_s)
\end{equation}
scheme-theoretically.  For general $s$, the curve $D_s$ is integral.  It
cannot equal $T_b(D_s)$, since that would give $T_b^*M\simeq M$, whereas
$b$ has order two and $K(M)\simeq(\Z/3\Z)^2$.  Hence the
intersection in \eqref{eq:16-zero-pullback} is proper of length $M^2=6$,
and therefore
\begin{equation}\label{eq:16-zero-length}
 \operatorname{length}Z_E(s)=3
\end{equation}
for general $s$.

We next prove that the determinant homomorphism
\[
 \delta:\bigwedge^2W\longrightarrow H^0(A,Q)
\]
is injective.  Otherwise, since $\dim W=3$, two independent sections of
$E$ would span a saturated rank-one subsheaf $R\subset E$.  Because
$E/R$ is torsion-free on the smooth surface $A$, the sheaf $R$ is
reflexive and hence a line bundle; moreover $h^0(A,R)\ge2$.  Put
$H:=\pi^*R$.  Projecting
$H\to\pi^*E=M\oplus T_b^*M$ to a nonzero summand, and translating by
$b$ if necessary, gives a nonzero map $H\to M$; here
$T_b^*H\simeq H$.  Thus both $H$ and $M\otimes H^{-1}$ are effective,
hence nef.

If $R^2>0$, then $R$ is ample and
$R^2=2h^0(A,R)\ge4$, so $H^2=2R^2\ge8$.  This contradicts
\[
 0\le (M-H)\cdot(M+H)=M^2-H^2.
\]
Consequently $R^2=H^2=0$.  Lemma~\ref{lem:standard}(iii) gives
$H\equiv nF$ for an elliptic curve $F\subset B$, where
$n=h^0(B,H)\ge2$.  Since $M-H$ is nef,
\[
 0\le(M-nF)^2=6-2n(M\cdot F).
\]
It follows that $M\cdot F=1$, and Lemma~\ref{lem:standard}(iv) would
make $B$ a product, a contradiction.  Thus $\delta$ is injective.

It follows that $\mathcal F$ has rank two.  The inclusion
$\operatorname{im}(\operatorname{ev})\hookrightarrow E$ induces
$\mathcal F\hookrightarrow E$ after taking reflexive hulls.  Since $A$
is a smooth surface, $\mathcal F$ is locally free; moreover, $W$
generates it in codimension one.  Taking determinants gives
\[
 \mathcal P\simeq Q(-\Delta)
\]
for an effective divisor $\Delta$, the fixed divisor of the determinant
system viewed in $|Q|$; after factoring it out, the induced system in
$|\mathcal P|$ has no divisorial base component.  If
$0\ne\sigma\in H^0(A,\mathcal F^\vee)$, composition with evaluation
gives a nonzero functional $\lambda:W\to\mathbf C$, because $W$
generates $\mathcal F$ generically.  Two
independent members of $\ker\lambda$ then have identically vanishing
determinant, contrary to the injectivity of $\delta$.  Hence
$H^0(A,\mathcal F^\vee)=0$, and $(\mathcal F,W)$ is a good pair.

Let
\[
 V:=\operatorname{im}\!\left(\bigwedge^2W
       \longrightarrow H^0(A,\mathcal P)\right).
\]
The injectivity just proved gives $\dim V=3$.  Moretti's
good-pairs correspondence and fiber formula
\cite[Proposition~1.5]{Moretti} produce the nondegenerate map
$\Phi=\varphi_V:A\dashrightarrow\PP(V^\vee)\simeq\PP(W)$.
The inclusion $\mathcal F\hookrightarrow E$ gives
$Z_{\mathcal F}(s)\subseteq Z_E(s)$ scheme-theoretically.  If $\Phi$ is
dominant, the same proposition and \eqref{eq:16-zero-length} therefore
give $\deg\Phi\le3$.  The Alzati--Pirola lower bound gives
$\deg\Phi\ge3$, and hence $\deg\Phi=3$.
\end{proof}

\subsection{The case of a pencil}

\begin{proposition}\label{prop:16-exceptional-image}
If the image of $\Phi$ is a curve, then $A$ is a product of elliptic
curves or admits a polarization of type $(1,2)$.  Consequently
$\irr(A)=3$.
\end{proposition}

\begin{proof}
Let $\Gamma\subset\PP(W)$ be the integral image curve.  It is
nondegenerate by Proposition~\ref{prop:16-good-pair}.  Put
$K:=\pi(K(M))\subset A$.  The kernels of $\pi$ and $\phi_M$ have
coprime orders, so restriction of $\pi$ identifies
\[
                 K(M)\simeq K\simeq(\Z/3\Z)^2.
\]
For $x\in K(M)$ put $y=\pi(x)$ and choose a theta-group lift
$T_x^*M\simeq M$.  The Cartesian translation square
$\pi\circ T_x=T_y\circ\pi$ gives
\[
                     T_y^*E\simeq\pi_*T_x^*M\simeq E.
\]
These isomorphisms are compatible with the theta-group action on $W$
and with evaluation.  They therefore preserve $\mathcal F$, its
determinant system, and its fixed divisor.  Thus $\Phi$ is
$K$-equivariant.  This base-change argument is the one underlying
\cite[proof of Theorem~2.2]{GrossiMoretti}, but it does not require the
auxiliary hypotheses of that theorem.
In particular, $\Gamma$ is $K$-invariant and $K$ acts faithfully on $\Gamma$ thanks to the non-degeneracy.

Let $\nu:C\to\Gamma$ be the normalization.  The $K$-action lifts
uniquely to $C$, and the lifted action is also faithful.
The curve $C$ cannot be rational.  Otherwise $K$ would embed in
$\operatorname{PGL}_2(\C)$, which is impossible. 
Hence, the rational map $\gamma: A\dashrightarrow C$ is actually a morphism to an elliptic curve $C$ and thus $\vert W^\vee\vert-\Delta$ is a base-point-free linear subsystem of $\vert\mathcal{P}\vert = \vert Q(-\Delta )\vert$.
Write its Stein factorization as
\[
  A\xrightarrow{p}C_0\xrightarrow{\eta}C,
\]
where $p$ has connected elliptic fibers and $\eta$ is an isogeny.  Put
\[
  N:=\eta^*\nu^*\cO_\Gamma(1),
  \qquad n:=\deg N.
\]
We then have $n=(\deg\eta)\deg\Gamma$.  Since an integral nondegenerate
plane curve of degree two is a smooth conic, whereas $C$ is elliptic, we
obtain $\deg\Gamma\ge3$ and hence $n\ge3$.

Let $F$ be a fiber of $p$, and put $m:=Q\cdot F$.  Since $\mathcal{P} = p^*N$, we have
\[
  \mathcal P\equiv nF,
  \qquad
  \Delta\equiv Q-nF.
\]
The effective divisor $\Delta$ is nef.  Using $Q^2=12$ and $F^2=0$, we
obtain
\[
  0\le\Delta^2=12-2nm.
\]
Since $n\ge3$ and $m$ is a positive integer, we must have $m=1$ or
$m=2$.

If $m=1$, Lemma~\ref{lem:standard}(iv), applied to the polarization $Q$,
shows that $A$ is a product of elliptic curves.  Suppose instead that
$m=2$.  The preceding inequality forces $n=3$ and $\Delta^2=0$, and we
have
\[
  \Delta\cdot F=(Q-3F)\cdot F=2.
\]
The divisor $\Delta+F$ is nef and satisfies
\[
  (\Delta+F)^2=\Delta^2+2\Delta\cdot F=4.
\]
It is therefore ample and has polarization type $(1,2)$.  In the first
case Proposition~\ref{prop:product} applies, and in the second case
Theorem~\ref{thm:type12} applies.  In either case, we conclude that
$\operatorname{irr}(A)=3$.
\end{proof}

\begin{theorem}\label{thm:type16}
If a complex abelian surface $A$ admits a polarization of type $(1,6)$,
then $\irr(A)=3$.
\end{theorem}

\begin{proof}
Choose $(B,M,b,\pi)$ as in Proposition~\ref{prop:16-cover}.  If $B$ is
a product, apply Lemma~\ref{lem:16-product-cover}.  Otherwise apply
Proposition~\ref{prop:16-good-pair}.  Its map is nondegenerate, so its
image is either $\PP^2$ or a nondegenerate curve.  In the first
case it has degree three; in the second,
Proposition~\ref{prop:16-exceptional-image} reduces to an already proved
case.  The Alzati--Pirola lower bound completes the proof.
\end{proof}

\begin{proof}[Proof of Theorem~\ref{thm:main}]
The necessary implication is Theorem~\ref{thm:necessary}.  Conversely,
the four sufficient cases are Theorems~\ref{thm:type11},
\ref{thm:type12}, \ref{thm:type13}, and \ref{thm:type16}, respectively.
Each supplies a dominant rational map of degree three to $\PP^2$.
\end{proof}

\begin{proof}[Proof of Corollary~\ref{cor:irr-classification}]
In any of the four polarization cases, Theorem~\ref{thm:main} and the
lower bound in \eqref{eq:universal-bounds} give $\irr(A)=3$.  If none of
the four polarizations exists, Theorem~\ref{thm:main} rules out a cubic
map, while \eqref{eq:universal-bounds} leaves only $\irr(A)=4$.
\end{proof}

\section{Appendix: some lattice results}
In this appendix, we establish several lattice-theoretic results used in the proof of Theorem \ref{thm:necessary}. The arguments are largely based on linear algebra.

A symmetric integral matrix $K$ is called \emph{even} if for any integral vector $v$, $v^tKv$ is even. An integer $n$ is called \emph{primitively represented} by $K$ if there is a primitive integral vector $v$ such that $n = v^tKv$.
\begin{lemma}\label{lem:binary-tensor}
Let $K$ be an even integral symmetric matrix of signature $(1,1)$ and let
$C\in\Sym_2(\Z)$.  Suppose that
\begin{equation}\label{hypothesis-binary}
     \tr(CK)=12,\qquad
 \det(K+KCK)\leq0,\qquad
 3K-KCK\preceq0.
\end{equation}
Then $K$ primitively represents one of $2,4,6,12$.
\end{lemma}

\begin{proof}
Put
\[
                   \Delta:=-\det K>0,
          \qquad \varepsilon:=\det(CK)=-\Delta\det C.
\]
For a $2\times2$ matrix, the trace and determinant determine the
characteristic polynomial.  Since $\tr(CK)=12$, we obtain
\begin{align}
 \det(K+KCK)&=-\Delta(13+\varepsilon),\label{eq:binary-det-plus}\\
 \det(3K-KCK)&=\Delta(27-\varepsilon).\label{eq:binary-det-minus}
\end{align}
The hypothesis in \eqref{hypothesis-binary}
implies
\begin{equation}\label{eq:epsilon-range}
                         -13\leq\varepsilon\leq27.
\end{equation}

Suppose first that $\Delta\geq28$.  Since $\varepsilon$ is a multiple
of $\Delta$, \eqref{eq:epsilon-range} gives $\varepsilon=0$.  The trace
condition excludes $C=0$, so $C$ has rank one.  Every nonzero integral
symmetric rank-one matrix can be written
\[
                              C=nvv^t
\]
with $v\in\Z^2$ primitive and $n\in\Z\setminus\{0\}$.  Hence
\[
                              n(v^tKv)=12.
\]
If $n<0$, then $v^tKv<0$.  For a nonzero real vector $w$ such that $v^tKw = 0$,
\[
 w^t(3K-KCK)w=3w^tKw>0,
\]
contrary to negative semidefiniteness.  Thus $n>0$, and the even
positive integer $v^tKv$ is a divisor of $12$.  This proves the result
in this case.

Assume now that $\Delta\leq27$.  Write
\[
 K=\begin{pmatrix}2a&b\\ b&2c\end{pmatrix},
 \qquad f(X,Y)=aX^2+bXY+cY^2.
\]
Writing $C=(c_{ij})$, the trace identity becomes
\[
                         ac_{11}+bc_{12}+cc_{22}=6.
\]
Thus the content $\gcd(a,b,c)$ of $f$ divides $6$.

Let $m$ be the least positive integer represented by $f$.  A vector
representing $m$ is primitive.  Complete it to an integral basis and
reduce the second basis vector modulo the first.  In the resulting
coordinates,
\[
                 f(X,Y)=mX^2+eXY+\kappa Y^2,
                 \qquad -m<e\leq m.
\]
If $\kappa>0$, then minimality gives $\kappa\geq m$, and the
discriminant would satisfy
\[
                  \Delta=e^2-4m\kappa\leq m^2-4m^2<0,
\]
a contradiction.  Thus $\kappa=-\nu$ with $\nu\geq0$, so
\begin{equation}\label{eq:binary-reduced}
                 f(X,Y)=mX^2+eXY-\nu Y^2.
\end{equation}

We claim that $m^2\leq\Delta=e^2+4m\nu$.  If $\nu=0$ and $|e|<m$,
choose $\eta\in\{\pm1\}$ with $e\eta=-|e|$; then
$0<f(1,\eta)=m-|e|<m$, contradicting minimality.  Hence in this case
$|e|=m$.  If $\nu>0$ and $\Delta<m^2$, then
\[
 4m\nu<(m-|e|)(m+|e|)\leq2m(m-|e|),
\]
so $\nu<(m-|e|)/2$.  With the same choice of $\eta$,
\[
                    0<f(1,\eta)=m-|e|-\nu<m,
\]
again a contradiction.  Therefore $m^2\leq\Delta\leq27$, and $m\leq5$.

If $m=1,2,$ or $3$, we are done.  If $m=4$ and $\nu=0$, then
$|e|=4$, so the form in \eqref{eq:binary-reduced} has content $4$,
contrary to the content condition.  If $m=4$ and $\nu>0$, then
$e^2+16\nu\leq27$, so $\nu=1$ and $|e|\leq3$.  If $|e|\leq2$, one
of $f(1,\pm1)$ equals $1$, $2$, or $3$; if $|e|=3$, the other sign
gives the value $6$.  If $m=5$ and $\nu=0$, the form has content $5$,
again impossible.  If $m=5$ and $\nu>0$, then $\nu=1$ and
$|e|\leq2$; one of $f(1,\pm1)=4\pm e$ is a positive value below $5$,
contradicting minimality.

Thus $f$ represents $1$, $2$, $3$, or $6$.  A vector representing any
of these values is primitive, since none is divisible by a nontrivial
square.  Therefore $K$ primitively represents $2$, $4$, $6$, or $12$.
\end{proof}

\begin{lemma}\label{lem:spectral}
Let $Q$ be a symmetric integral matrix with signature $(1,2)$ and let $B$ be a symmetric integral matrix with
$\indp(B)=1$.  Assume
\[
                              B-3Q^{-1}\succeq0.
\]
Then every eigenvalue of $BQ$ is real and nonnegative, at most one
eigenvalue is greater than $3$, and the eigenvalue zero is semisimple.
\end{lemma}

\begin{proof}
Write
\[
                         B-3Q^{-1}=XX^t,
\]
where $X$ has full column rank, possibly zero, and put
$C=X^tQX$.  Consider
\[
 \mathcal H=
 \begin{pmatrix}
   3Q^{-1}&X\\
   X^t&-I
 \end{pmatrix}.
\]
Note that the matrix \(\mathcal H\) is congruent to
\[
\begin{pmatrix}
B & 0\\
0 & -I
\end{pmatrix}
\quad \text{ and } \quad
\begin{pmatrix}
3Q^{-1} & 0\\
0 & -I-\frac{1}{3}C
\end{pmatrix}.
\]
Hence we have 
$$  1 = \indp(B) = \indp(\mathcal H)
   =\indp(3Q^{-1})+\indp(-I-C/3)
   =1+n_-(I+C/3).$$
Here $n_-$ denotes the negative index of a real symmetric form.
Consequently $C\succeq-3I$.  Moreover,
$\indp(C)\leq \indp(Q)=1$.

Set $Y=X^tQ$.  Then
\[
                         BQ=3I+XY,
                  \qquad YX=C.
\]
For every nonzero eigenvalue, the matrices $XY$ and $YX$ have the
same generalized eigenspaces and the same Jordan blocks.
Since $C$ is symmetric, its spectrum is real, lies in $[-3,\infty)$,
and contains at most one positive eigenvalue.  It follows that the
eigenvalues of $BQ$ are nonnegative, and at most one exceeds $3$.
Finally, an eigenvalue zero of $BQ$ corresponds to the nonzero
eigenvalue $-3$ of the symmetric matrix $C$, so its Jordan blocks have
size one.
\end{proof}

\begin{lemma}\label{lem:ternary}
Every even integral symmetric matrix of signature $(1,2)$ and determinant at
most $54$ primitively represents one of $2,4,6,12$.
\end{lemma}

\begin{proof}
Suppose, to the contrary, that there exists an integral symmetric matrix \(Q\) of signature \((1,2)\) and determinant at most \(54\) that represents none of \(2,4,6,12\). For any vectors \(v,w\), we write
$$
(v,w):=v^tQw.
$$
  
Let $m:=\min\{(v,v)>0\;\vert\; v\in \Z^3\}$ and let $h$ be an integral vector such that $m = (h,h)$.
Such a vector exists because the positive cone contains a rational,
hence after scaling an integral, vector.  
Hence $h$ is primitive and by assumption, we have
\begin{equation}\label{eq:m-possibilities}
                         m=8,\ 10,\quad\text{or}\quad m\geq14.
\end{equation}

Define the positive definite quadratic form
\[
                  H_h[x]=\frac{2(h,x)^2}{m}-(x,x).
\]
In the orthogonal decomposition
$\Lambda_\R=\R h\oplus h^\perp$, this form agrees with $Q$ on
$\R h$ and with $-Q$ on $h^\perp$.  The corresponding involution has
determinant one, so
\[
                              \det H_h=\det Q.
\]
The ellipsoid $H_h[x]\leq t$ has volume
\[
                    \frac{4\pi}{3}\frac{t^{3/2}}{\sqrt{\det Q}}.
\]
Because $\det Q\leq54<6\pi^2$, choose
\[
     \left(\frac{6}{\pi}\right)^{2/3}(\det Q)^{1/3}<t<6.
\]
Minkowski's convex-body theorem gives a nonzero integral vector $x$ with
$H_h[x]\leq t<6$; see \cite[Chapter~III, \S2]{Cassels}.

Replace $x$ by $x-nh$, choosing $n$ nearest to $(h,x)/m$.  In the
orthogonal decomposition used above, this does not increase $H_h$ and
ensures
\begin{equation}\label{am_inequality}
    a=|(h,x)|\leq m/2.
\end{equation}

The new vector is nonzero, since $H_h[nh] = n^2m\geq 8$ by \eqref{eq:m-possibilities} if $n$ is non-zero but $H_h[x]<6$. Note that  $(x,x)\leq 0$: otherwise
minimality gives $(x,x)\geq m$, while
\[
                  H_h[x]\leq m/2-(x,x)<0,
\]
contrary to positive definiteness.  Write
$$
                              (x,x)=-2s.
$$
Since $H_h[x]<6$, we have $s\in\{0,1,2\}$.  Choose the sign of $x$ so
that $(h,x)=-a$.  The case $a=s=0$ is impossible because $Q\vert_{h^\perp}$ is
negative definite.  Also $h+x\neq0$, and
\[
                         (h+x)^2=m-2a-2s<m.
\]
If this square were positive, it would contradict the minimality of
$m$.  Hence
\begin{equation}\label{eq:minimality-inequalities}
                  m\leq2(a+s),
          \qquad \frac{2a^2}{m}+2s<6.
\end{equation}
Together with \eqref{eq:m-possibilities} and \eqref{am_inequality}, these inequalities leave
exactly
\[
\begin{array}{c|c}
 s&(m,a)\\ \hline
 0&(8,4),(10,5),\\
 1&(8,3),(10,4),\\
 2&(8,2),(10,3).
\end{array}
\]
In every case $(h+x)^2=0$.  Dividing $h+x$ by its integral content,
we obtain a primitive isotropic vector $u$.

Let
\[
                   d=\operatorname{div}(u)
                     =\gcd\{(u,y):y\in\Lambda\}.
\]
The map $y\mapsto(u,y)/d$ from $\Z^3$ to $\Z$ is primitive.
Its kernel is saturated, and $u$ extends to a basis $u,w$ of that
kernel.  Choose $v$ mapping to one.  In the basis $u,v,w$, the Gram
matrix has the form
\begin{equation}\label{eq:isotropic-normal-form}
 \begin{pmatrix}
  0&d&0\\
  d&2A&b\\
  0&b&-2k
 \end{pmatrix},
 \qquad k\geq1.
\end{equation}
Indeed, $Q$ restricted on $u^\perp/\R u$ is negative definite, which gives the last
entry.  Taking determinants in \eqref{eq:isotropic-normal-form} gives
\[
                         \det Q=2kd^2\leq54,
\]
so $d\leq5$.  

Now consider the primitive vector $nu+v+zw$ and its half-square
\begin{equation}\label{eq:half-square-residue}
 \frac{(nu+v+zw, nu+v+zw)}{2}=dn+A+bz-kz^2.
\end{equation}
If $d\leq3$, the residues of $1$, $2$, and $3$ cover every residue
class modulo $d$.  Fixing $z=0$ and choosing $n$ in
\eqref{eq:half-square-residue} therefore gives one of these three
values.  If $d=4$ or $5$, the determinant bound forces $k=1$.
Modulo $4$, the polynomial $A+bz-z^2$ cannot vanish identically: its
values at $z=0$ and $z=1$ would force $A\equiv0$ and $b\equiv1$, but
then its value at $z=2$ is $2$.  Hence it assumes one of the residues
$1,2,3$.  Modulo $5$, completing the square shows that
$A+bz-z^2$ has exactly three values; it cannot have image contained in
the two-element set $\{0,4\}$, so again it assumes one of $1,2,3$.
Choosing $n$ in \eqref{eq:half-square-residue} produces a vector of
half-square $1$, $2$, or $3$.  This contradiction proves the lemma.
\end{proof}

\begin{proposition}\label{prop:rank-three-arithmetic}
Let $\Lambda = \Z^3$, let $Q$ be an even integral symmetric matrix of signature $(1,2)$ on $\Lambda$, and let
$B$ be an integral symmetric matrix.  If
\begin{equation}\label{eq:arithmetic-hypotheses}
 \tr(BQ)=12,\qquad \indp(B)=1,\qquad 3Q-QBQ\preceq0,
\end{equation}
then $Q$ primitively represents one of $2,4,6,12$.
\end{proposition}

\begin{proof}
The last inequality in \eqref{eq:arithmetic-hypotheses} is equivalent
to
\[
                              B-3Q^{-1}\succeq0,
\]
because
\[
                  3Q-QBQ=-Q(B-3Q^{-1})Q.
\]
Apply Lemma~\ref{lem:spectral} to $T=BQ$.  Its eigenvalues are
nonnegative, at most one exceeds $3$, zero is semisimple, and their
sum is $12$.

Let
\[
 M=\sat\bigl(\operatorname{im}(B:\Lambda^*\to\Lambda)\bigr),
          \qquad s=\rk B.
\]
The trace identity gives $s\geq1$.
Since $M$ is saturated, a basis of $M$ extends to a basis of $\Lambda$.
In such a basis, symmetry of $B$ gives block matrices
\begin{equation}\label{eq:block-forms}
 B=\begin{pmatrix}A&0\\0&0\end{pmatrix},
 \qquad
 Q=\begin{pmatrix}K&R\\R^t&S\end{pmatrix},
 \qquad
 BQ=\begin{pmatrix}AK&AR\\0&0\end{pmatrix}.
\end{equation}
Because $Q$ is invertible, $\rk(BQ)=s$, so the geometric multiplicity
of zero is $3-s$.  Semisimplicity makes its algebraic multiplicity
also $3-s$.  Thus $AK$ is nonsingular, and all its eigenvalues are
positive.

The nonsingular matrix $A$ has signature $(1,s-1)$, so
$\operatorname{sign}(\det A)=(-1)^{s-1}$.  Positivity of the eigenvalues
of $AK$ gives $\det(AK)>0$, hence
$\operatorname{sign}(\det K)=(-1)^{s-1}$.  The restriction $K$ of the
form $Q$ has positive index at most one.  Since it is nondegenerate,
the determinant sign forces $K$ to have signature $(1,s-1)$.

If $s=1$, write $A=(c)$.  Then $c>0$, and
\[
                         12=\tr(AK)=cK_{11}.
\]
The primitive generator of the saturated rank-one lattice $M$ has
positive even self-pairing $K_{11}\in\{2,4,6,12\}$.

If $s=2$, the two eigenvalues of $AK$ are positive.  Therefore
\[
 \det(K+KAK)=\det K\det(I+AK)<0.
\]
The upper-left principal block of $3Q-QBQ\preceq0$ is
$3K-KAK\preceq0$, and $\tr(AK)=12$.  Lemma~\ref{lem:binary-tensor}, applied to $(K,A)$, gives a primitive vector of $M$ whose self-pairing is one of the desired values. It remains primitive in $\Lambda$ because $M$ is saturated.

Finally, suppose $s=3$.  All three eigenvalues of $BQ$ are positive.
Choose the two that are at most three and denote them by $x,y$.  The
third is $12-x-y$, and
\[
                         \det(BQ)=xy(12-x-y).
\]
On the square $0\leq x,y\leq3$, this function is increasing in each
variable, since
\[
 \frac{\partial}{\partial x}\bigl(xy(12-x-y)\bigr)
    =y(12-2x-y)\geq3y,
\]
and similarly in $y$.  Hence
\[
                             \det(BQ)\leq54.
\]
Both $B$ and $Q$ have signature $(1,2)$, so $\det B$ and $\det Q$ are
positive integers.  Since $\det(BQ)=\det B\det Q$, we obtain
\[
                             0<\det Q\leq54.
\]
Lemma~\ref{lem:ternary} completes the proof.
\end{proof}

\end{document}